\documentclass[11pt]{amsart}

\usepackage{subfiles}
\usepackage{amssymb,amsfonts}
\usepackage{graphicx}
\usepackage{caption,subcaption}
\usepackage{color}
\usepackage{tikz}
\usetikzlibrary{calc}
\usetikzlibrary{arrows}
\usetikzlibrary{decorations.markings}
\usetikzlibrary{shapes}

\usepackage{fullpage}

\def\VEC#1#2#3{#1_{#2},\ldots,#1_{#3}}
\def\sset{\subseteq}
\def\cS{\mathcal{S}}

\def\nin{\not\in}

\newcommand\xp{\operatorname{\hspace{-1pt}\mathaccent"017{\phantom{cc}}\hspace{-8pt}ch}}
\newcommand\mad{\operatorname{Mad}}
\def\extend#1{$#1$-extendable}
\usepackage{graphicx}

\newcommand\ch{\operatorname{ch}}
\newcommand{\floor}[1]{\left\lfloor#1\right\rfloor}
\def\set#1{\left\{#1\right\}}

\newtheorem*{thmgeneral}{Theorem~\ref{lem:contract}}

\newtheorem{theorem}{Theorem}[section]

\newtheorem{lemma}[theorem]{Lemma}

\theoremstyle{definition}

\theoremstyle{definition}
\newtheorem{definition}[theorem]{Definition}
\newtheorem{defn}[theorem]{Definition}

\newtheorem{ex}[theorem]{Example}
\theoremstyle{remark}

\newtheorem{rem}[theorem]{Remark}
\numberwithin{equation}{section}
\theoremstyle{plain}

\newtheorem{question}[theorem]{Question}
\newtheorem{cor}[theorem]{Corollary}

\begin{document}

\tikzset{
    Eedge/.style={
    line width = 3pt,gray!65
    }}

\title{Dynamic coloring parameters for graphs with given genus}
\author[S. Loeb]{Sarah Loeb}
\author[T. Mahoney]{Thomas Mahoney}
\author[B. Reiniger]{Benjamin Reiniger}
\author[J. Wise]{Jennifer Wise}
\thanks{University of Illinois at Urbana--Champaign\\
Research of the authors supported in part by NSF grant DMS 08-38434,
``EMSW21-MCTP: Research Experience for Graduate Students''.\\
emails: sloeb2@illinois.edu, tmahoney@emporia.edu, reiniger@ryerson.ca, jiwise2@illinois.edu}
\date{\today}

\begin{abstract}
A proper vertex coloring of a graph $G$ is \emph{$r$-dynamic}
if for each $v\in V(G)$, at least $\min\{r,d(v)\}$ colors appear in $N_G(v)$.
In this paper we investigate $r$-dynamic versions of 
coloring, list coloring, and paintability.
We prove that planar and toroidal graphs are 3-dynamically
10-colorable, and this bound is sharp for toroidal graphs.
We also give bounds on the minimum number of colors needed for any $r$
in terms of the genus of the graph: 
for sufficiently large $r$, 
every graph with genus $g$ is $r$-dynamically $((r+1)(g+5)+3)$-colorable
when $g\leq2$ and $r$-dynamically $((r+1)(2g+2)+3)$-colorable when $g\geq3$. Furthermore, each of these upper bounds for $r$-dynamic $k$-colorability
also holds for $r$-dynamic $k$-choosability and for $r$-dynamic $k$-paintability.
We develop a method to prove that certain configurations are reducible
for each of the corresponding $r$-dynamic parameters.
\end{abstract}

\maketitle


\section{Introduction}
For a graph $G$ and positive integer $r$,
an \emph{$r$-dynamic coloring} of $G$ is a proper vertex coloring such that for each $v\in V(G)$,
at least $\min\{r,d(v)\}$ distinct colors appear in $N_G(v)$.
The \emph{$r$-dynamic chromatic number}, denoted $\chi_r(G)$,
is the minimum $k$ such that $G$ admits an $r$-dynamic $k$-coloring.
Montgomery~\cite{montgomery} introduced 2-dynamic coloring and the generalization to $r$-dynamic coloring. 

List coloring was introduced independently by Vizing~\cite{vizing} and by Erd\H{o}s, Rubin, and Taylor~\cite{ERT}.
A \emph{list assignment} $L$ for $G$ assigns to each vertex $v$ a list $L(v)$ of permissible colors.
Given a list assignment $L$ for a graph $G$, if a proper coloring $\phi$ can be chosen so that $\phi(v)\in L(v)$ for all $v\in V(G)$,
then $G$ is \emph{$L$-colorable}.
The \emph{choosability} of $G$ is the least $k$ such that $G$ is $L$-colorable for any list assignment $L$ satisfying $|L(v)|\ge k$ for all $v\in V(G)$.
We consider the $r$-dynamic version of this parameter.
For further work, see~\cite{AGJ,JKOW,KMW}.
A graph $G$ is \emph{$r$-dynamically $L$-colorable} when an $r$-dynamic coloring can be chosen from the list assignment $L$.
The \emph{$r$-dynamic choosability} of $G$, denoted $\ch_r(G)$,
is the least $k$ such that $G$ is \emph{$r$-dynamically $L$-colorable} for every list assignment $L$ satisfying $|L(v)|\ge k$ for all $v\in V(G)$.

Zhu~\cite{zhu} and Schauz~\cite{schauz} independently introduced an online version of choosability,
which is modeled by the following game.
\begin{definition}\label{def:paint}
Suppose $G$ is a graph and that each vertex $v\in V(G)$ is assigned a positive number $f(v)$ of \emph{tokens}.
The \emph{$f$-paintability game} is played by two players: Lister and Painter.
On the $i$th round, Lister \emph{marks} a nonempty set of uncolored vertices;
each marked vertex loses one token.
Painter responds by choosing a subset of the marked set that forms an independent set in the graph and assigning color $i$ to each vertex in that subset.
Lister wins the game by marking a vertex with no tokens,
and Painter wins by coloring all vertices.

We say $G$ is \emph{$f$-paintable} when Painter has a winning strategy in the $f$-paintability game.
When $G$ is $f$-paintable and $f(v)=k$ for all $v\in V(G)$,
we say that $G$ is \emph{$k$-paintable}.
The least $k$ such that $G$ is $k$-paintable is the \emph{paint number} (or \emph{online choice number}) of $G$, denoted by $\xp(G)$.

In the $f$-paintability game, Painter's goal is to generate a proper coloring of the graph.
We say that a graph $G$ is \emph{$r$-dynamically $k$-paintable} when
Painter has a winning strategy that produces an $r$-dynamic coloring of $G$ when all vertices have $k$ tokens.
The least $k$ such that Painter can accomplish this is the \emph{$r$-dynamic paint number}, denoted by $\xp_r(G)$.
\end{definition}

The \emph{square} of a graph $G$, denoted $G^2$, is the graph resulting from adding an edge between every pair of vertices of distance 2 in $G$. For any graph $G$, it is clear that
\[
\begin{array}[b]{r@{\hspace{3pt}}r@{\hspace{3pt}}r@{\hspace{3pt}}r@{\hspace{3pt}}r@{\hspace{3pt}}r@{\hspace{3pt}}r@{\hspace{3pt}}r@{\hspace{3pt}}r@{\hspace{3pt}}r@{\hspace{3pt}}r@{\hspace{3pt}}r@{\hspace{3pt}}r@{\hspace{3pt}}r@{\hspace{3pt}}r@{\hspace{3pt}}}
\chi(G) & = & \chi_1(G) & \leq & \chi_2(G) & \leq & \dotsb & \leq & \chi_{\Delta(G)}(G) & = & \dotsb & = & \chi(G^2),\\
\ch(G) & = & \ch_1(G) & \leq & \ch_2(G) & \leq & \dotsb & \leq & \ch_{\Delta(G)}(G) & = & \dotsb & = & \ch(G^2),&\qquad (1)\\
\xp(G) & = & \xp_1(G) & \leq & \xp_2(G) & \leq & \dotsb & \leq & \xp_{\Delta(G)}(G) & = & \dotsb & = & \xp(G^2),
\end{array}
\]
and that $\chi_r(G)\leq\ch_r(G)\le\xp_r(G)$ for all $r$.
Thus we can think of $r$-dynamic coloring as bridging the gap between coloring a graph and coloring its square.


Wegner~\cite{wegner} conjectured bounds for the chromatic number of squares of planar graphs given their maximum degree.
For a graph $G$ with $\Delta(G)\le3$, proper colorings of $G^2$ and 3-dynamic colorings of $G$ are equivalent.
Thomassen~\cite{thomassen7} proved Wegner's conjecture for maximum degree 3,
showing that $\chi_3(G)\leq7$ for any planar subcubic graph $G$.
Cranston and Kim~\cite{CK} studied the list coloring version,
and proved that when $G$ is a planar subcubic graph,
$\ch_3(G)\le 7$ if the girth is at least 7 and $\ch_3(G)\le 6$ if the girth is at least 9.

Thomassen~\cite{thomassen5} proved that planar graphs are 5-choosable,
and Voigt~\cite{voigt} proved sharpness.
Schauz~\cite{schauz} further proved that planar graphs are 5-paintable.
Kim, Lee, and Park~\cite{KLP} proved that planar graphs are actually 2-dynamically 5-choosable. 
Their proof involves showing that every planar graph has a planar supergraph with an edge in the neighborhood of every vertex.
They then invoke Thomassen's result that planar graphs are 5-choosable to obtain their result.
By using Schauz's result that planar graphs are 5-paintable instead, 
Kim, Lee, and Park's result is strengthened to the following corollary.

\begin{cor}\label{cor:5paint}
If $G$ is a planar graph, then $\xp_2(G)\le5$.
\end{cor}

Heawood~\cite{heawood} proved that for $g>0$,
graphs of (orientable) genus $g$ are $(h(g)-1)$-degenerate and hence $h(g)$-colorable, where
\[ h(g)=\floor{\frac{7+\sqrt{1+48g}}{2}}. \]
Because $(k-1)$-degenerate graphs are $k$-paintable, this also shows that graphs with genus $g$ are $h(g)$-paintable.
Chen et al.~\cite{CFLSS} proved that such a graph is 2-dynamically $h(g)$-choosable.
Mahoney~\cite{mahoney} strenthened their result to prove that such a graph is 2-dynamically $h(g)$-paintable.

Our main results are on the 3-dynamic chromatic number, choosability, and paint numbers for planar and toroidal graphs.  We will call a graph \emph{toroidal} if it can be drawn on the torus without crossing edges; in particular, we consider planar graphs to also be toroidal.

\begin{theorem} \label{thm:maintorus}
If $G$ is a toroidal graph, then $\chi_3(G)\le\ch_3(G) \le \xp_3(G) \le 10$. 
\end{theorem}

Theorem~\ref{thm:maintorus} is sharp: the Petersen graph $P$ has maximum degree 3 and diameter 2, so $\chi_3(P)=\chi(P^2)=\chi(K_{10})=10$.

Our proofs use the Discharging Method.
A \emph{configuration} in a graph is a set of vertices that satisfies some specified condition,
for example, a condition on the degrees or adjacencies of the vertices in the configuration.
We say that a configuration in a graph is \emph{reducible}
for a graph property if it cannot occur in a minimal graph failing that property.
We say that a partial coloring of a graph $G$ \emph{extends}
if the uncolored vertices can be assigned colors so that the coloring for all of $V(G)$ is an $r$-dynamic coloring of $G$. 

As an immediate corollary of Theorem~\ref{thm:maintorus}, we have that: 

\begin{cor} \label{cor:mainplane}
If $G$ is a planar graph, then $\chi_3(G)\le\ch_3(G) \le \xp_3(G)\le 10$. 
\end{cor}

We do not believe that Corollary~\ref{cor:mainplane} is sharp; 
the proof we give relies heavily on showing that the configuration consisting of two adjacent vertices both having degree 3 is reducible for $3$-dynamic $10$-paintability.
An example of a planar graph $G$ with $\chi_3(G)=7$ is the graph obtained from $K_4$ by subdividing the three edges incident to one vertex, shown in Figure~\ref{fig:K4}.
Note that $G$ has maximum degree 3 and diameter 2, so $\chi_3(G)=\chi(G^2)=\chi(K_7)$.

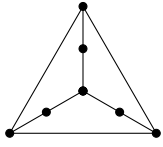
\begin{figure} \centering
\begin{tikzpicture}[scale=.75]
\path (0,0) coordinate (V);
\foreach \i in {1,2,3} {\foreach \j in {1,2} {
\path (V)++(90+120*\i:.75*\j cm) coordinate (X\j\i);}}
\foreach \i in {1,2,3} {
\draw (V) -- (X2\i);
\fill[draw=black] (V) circle (2pt); 
\fill[draw=black] (X1\i) circle (2pt); 
\fill[draw=black] (X2\i) circle (2pt);}
\draw (X21) -- (X22) -- (X23) -- cycle;
\end{tikzpicture}
\caption{Example for Corollary~\ref{cor:mainplane}}
\label{fig:K4}
\end{figure}

In Section~\ref{sec:framework}, we provide a reduction procedure and a winning strategy for Painter.
When any $r$-dynamic coloring of $G'$ may be extended to vertices of $V(G)-V(G')$ and form an $r$-dynamic coloring of $G$,
we give additional conditions under which Painter can win the $r$-dynamic paintability game on $G$.

In Section~\ref{sec:reducibleconfigs}, we show that several configurations are reducible for $3$-dynamic $10$-paintability.
Our reduction procedure from Section~\ref{sec:framework} also implies that each configuration is reducible for
$3$-dynamic $10$-choosability and $3$-dynamic $10$-colorability.
In Section~\ref{sec:discharging}, we complete the proof of Theorem~\ref{thm:maintorus}
by using the Discharging Method to show that the configurations listed in Section~\ref{sec:reducibleconfigs}
form a set that is unavoidable by a toroidal graph.

Finally, in Section~\ref{sec:rdynamic}, we consider graphs with higher genus.
Let $\gamma(G)$ denote the minimum genus of a surface on which $G$ embeds.

\begin{theorem}\label{lem:contract}
Let $G$ be a graph, and let $g=\gamma(G)$.
\begin{enumerate}
\item If $g \le 2$ and $r \ge 2g + 11$, then $\xp_r(G) \le (g+5)(r+1)+3$.
\item If $g \ge 3$ and $r \ge 4g + 5$, then $\xp_r(G) \le (2g+2)(r+1)+3$.
\end{enumerate}
\end{theorem}

Even though Theorem~\ref{lem:contract} requires a lower bound on $r$, by (1), it gives an upper bound on $\xp_r(G)$ for all $r$.

\section{Framework for Reducibility Arguments}\label{sec:framework}

In Sections~\ref{sec:reducibleconfigs}
 and~\ref{sec:rdynamic}, 
we consider several reducibility arguments.  
From a high-level view, standard reducibility arguments involve extending a coloring of only part of a graph to the whole graph.
Remark~\ref{rem:extend} provides a general method for making reducibility arguments in the
paintability, choosability, and classical graph coloring contexts. 

Remark~\ref{rem:extend} is phrased in terms of the strongest parameter, paintability.
We must address the issue that in the $f$-paintability game,
Painter cannot ``wait until the end'' to color a particular part of the graph.
Given a graph $G$ and $S\sset V(G)$, instead of extending a coloring of $G-S$ to $S$ as in colorability or choosability,
there will be rounds when Painter must color vertices in both $S$ and $G-S$.
Remark~\ref{rem:extend} gives a strategy for Painter to color vertices in $S$
on the same round as vertices in $G-S$ when certain conditions are met.

\smallskip
Given a graph $G$ and a vertex $v\in V(G)$,
let the \emph{closed neighborhood} of $v$, denoted $N_G[v]$, be $N_G(v)\cup\set{v}$.
Additionally, for a set $S\sset V(G)$, let $N_G[S]=\bigcup_{v\in S}N_G[v]$ and $N_G(S)=N_G[S]-S$.
Let $[k]=\set{1,2,\dots,k}$.
We say a vertex $v$ is \emph{rejected} in the Lister/Painter game each time $v$ is marked by Lister,
but not colored by Painter.

Let $G$ and $G'$ be graphs satisfying $V(G')\sset V(G)$.
We say that $G'$ is \emph{$r$-extendable to $G$}
if every $r$-dynamic coloring of $G'$ extends to an $r$-dynamic coloring of $G$.
It is not necessary for $G'$ to be a subgraph of $G$.
In fact, Example~\ref{ex:notsubgraph} shows one instance where it is necessary for $G'$ to contain edges that are not in $E(G)$.

\begin{ex}\label{ex:notsubgraph}
Let $r>1$ and let $G$ be a graph.
Suppose that $x\in V(G)$ has degree 2 with $N(x)=\set{y,z}$ and that $yz\nin E(G)$.
Let $G'=G-x$.
Any $r$-dynamic coloring of $G'$ that gives $y$ and $z$ the same color does not extend to an $r$-dynamic coloring of $G$
since $N(x)$ will not contain at least 2 colors.

We overcome this problem by letting $G'=(G\cup\set{yz})-x$.
Any $r$-dynamic coloring of $G'$ gives $y$ and $z$ different colors, so $N(x)$ always receives two colors.
However, while an $r$-dynamic coloring of $G'$ gives $N_{G'}(w)$ at least $\min\set{r,d(w)}$ colors for $w\in\set{y,z}$, 
it may be the case that $N_G(w)-v$ receives only $\min\set{r,d(w)}-1$ colors for $w\in\set{y,z}$
since $y$ and $z$ are not neighbors of each other in $G$.
By forcing $x$ to avoid $\min\set{r,d(w)}-1$ colors used on $N_G(w)-v$ for $w\in\set{y,z}$,
we can ensure that an $r$-dynamic coloring of $G'$ extends to an $r$-dynamic coloring of $G$.
\hfill$\blacksquare$
\end{ex}

When $G'$ is $r$-extendable to $G$, a winning strategy for Painter on $G'$
may be combined with a strategy on $G-G'$ to produce an $r$-dynamic coloring of $G$.
To state this process more formally, we give the following definition.

\begin{defn}\label{def:first}
Let $G'$ be a graph that is $r$-extendable to $G$.
We say that Painter plays a \emph{$G'$-first strategy} in the $r$-dynamic $k$-paintability game on $G$ if the following conditions are satisfied:
\begin{itemize}
\item $G'$ is $r$-dynamically $k$-paintable.
\item For any marked set $M$ in the game on $G$, Painter's response $D$ contains a winning response $D'$ to the marked set $M\cap V(G')$ in the $r$-dynamic $k$-paintability game on $G'$.
\item At the end of each round, the colored vertices form a partial coloring of $G$ that extends to an $r$-dynamic coloring of $V(G)$.
\hfill$\blacksquare$
\end{itemize}
\end{defn}

If Painter wins the $r$-dynamic $k$-paintability game on $G'$ by playing a $G'$-first strategy $\cS$,
then Painter is always winning in the auxiliary game being played on $G'$, 
regardless of what is being marked and colored in $V(G)-V(G')$.
Let $T=V(G)-V(G')$.
By giving an upper bound on how many times Painter rejects each vertex of $T$ by playing according to $\cS$,
we show that Painter has a winning strategy in the game on $G$.
The following remark is stated more generally in~\cite{mahoney}.

\begin{rem}\label{rem:extend}
Given graphs $G$ and $G'$,
where $G'$ is both $r$-dynamically $k$-paintable and $r$-extendable to $G$,
let $T=V(G)-V(G')$.
If Painter has a $G'$-first strategy on $G$ such that each $v\in T$ is rejected at most $k$ times,
then $G$ is $r$-dynamically $k$-paintable.
\end{rem}

Remark~\ref{rem:extend} holds because it describes the conditions 
for an inductive strategy to succeed for Painter.
We may relax ``$k$-paintable'' in both the hypothesis and conclusion of Remark~\ref{rem:extend}
to be ``$k$-choosable'' or ``$k$-colorable'' and still obtain the desired result.
Thus Remark~\ref{rem:extend} serves as a general tool
for proving upper bounds on the $r$-dynamic chromatic, choice, and paint numbers.

We now give an application of Remark~\ref{rem:extend}.
Kim and Park~\cite{KP} use the Discharging Method to give results about 2-dynamic choosability of sparse graphs.
Their proof omits proving that 1-vertices are reducible for 2-dynamic 4-choosability.
We give a complete proof and strengthen their result by showing that 
1-vertices and 2-vertices adjacent to $3^-$-vertices are reducible for 2-dynamic 4-paintability. 

The \emph{maximum average degree} of a graph $G$, denoted $\mad(G)$,
is $\max_{H\sset G}\frac{2|E(H)|}{|V(H)|}$.


\begin{lemma}[\cite{KP}]\label{lem:KPunavoid}
The following configurations form an unavoidable set for $2$-dynamic $4$-paintability
when $G$ is a planar graph with girth at least 7 or
a graph other than $C_5$ satisfying $\mad(G)<\frac{8}{3}$:
\begin{enumerate}
\item a $1$-vertex,
\item a $2$-vertex adjacent to a $3^-$-vertex.
\end{enumerate}
\end{lemma}

Kim and Park~\cite{KP} use the choosability version of Lemma~\ref{lem:KPunavoid}
to show that graphs with maximum average degree less than $8/3$ and
planar graphs of girth at least 7 are 2-dynamically 4-choosable.
However, $\ch_2(C_5) = 5$.
We use Lemma~\ref{lem:KPunavoid} to correct and strengthen these results. 


\begin{theorem}\label{thm:KPstrengthen}
Let $G$ be a connected graph other than $C_5$.
If $\mad(G)<\frac83$ or $G$ is a planar with girth at least 7,
then $\xp_2(G)\le4$.
\end{theorem}

\begin{proof}
If $\Delta(G)\le2$, then $\xp_2(G)=\xp(G^2)$, and so $\Delta(G^2)\le4$.
By the paintability analogue of Brooks' Theorem~\cite{HKS},
we have that $\xp(G^2)\le\Delta(G^2)\le4$ unless $G^2$ is a complete graph of order at least 5.
Since $G\ne C_5$, we conclude that $\xp_2(G)\le4$.

Thus $\Delta(G)>2$,
and it suffices by Lemma~\ref{lem:KPunavoid} to show that a 1-vertex
and a 2-vertex adjacent to a $3^-$-vertex are reducible for 2-dynamic 4-paintability.
To prove reducibility, we may assume that $G$ is a minimal graph failing to be 2-dynamically 4-paintable.
In each case, we define $G'$, which is 2-dynamically 4-paintable by minimality.
We then assume that Painter plays according to a $G'$-first strategy
and use Remark~\ref{rem:extend} to show that Painter wins the game on $G$.

\textbf{Case 1:} Let $v$ be a 1-vertex.\\
Let $G'=G-v$.
Since $K_2$ is trivially 2-dynamic 4-paintable,
we may assume that $N_G(v)=\set{u}$ and $d(u)>1$.
If $v$ is marked on the round in which $u$ is colored, then Painter rejects $v$.
Painter also rejects $v$ if there is a round in which all other neighbors of $u$ are being colored.
Thus $v$ is rejected at most two times.

\textbf{Case 2a:} Let $d(u)=d(v)=2$ where $uv\in E(G)$.\\
Since $\Delta(G)>2$, we may select $u$ so that its other neighbor $u'$ has degree at least 3.
Let $v'$ be the neighbor of $v$ other than $u$.
Let $G'=G-\set{u,v}$.
Painter rejects $u$ on rounds in which $v$ or $u'$ is colored.
Painter also rejects $u$ in a round if $v$ is not yet $2$-dynamic but $v'$ is being colored.
Painter rejects $v$ on rounds in which $v'$ is colored.
Painter also rejects $v$ in a round if $u$ is not yet $2$-dynamic but that $u'$ is being colored.
Lastly, Painter rejects $v$ in a round if $v'$ is not yet $2$-dynamic but that another neighbor of $v'$ is being colored.
Thus $u$ and $v$ are each rejected at most three times.

\textbf{Case 2b:} Let $d(u)=3$ and $d(v)=2$ where $uv\in E(G)$.\\
Let $T$ be the set of $2$-neighbors of $u$, let $G'=G-u-T$, and let $v'$ be the neighbor of $v$ other than $u$.
Note that $T$ is nonempty since $v\in T$.
Case 2a implies that for $w\in T$, the neighbor of $w$ other than $u$ has degree at least 3.
Painter rejects $u$ whenever the other neighbor of a vertex in $T$ is colored.
For $w\in T$, let $w'$ be the neighbor of $w$ other than $u$.
Painter rejects each $w\in T$ whenever $u$, $w'$, or $v'$ is colored.
Thus each vertex in $T\cup{u}$ is rejected at most three times.

In each case, Remark~\ref{rem:extend} implies that the configuration is reducible for 2-dynamic 4-paintability.
\end{proof}

\section{Reducibility Arguments}\label{sec:reducibleconfigs}


We use ``reducible'' in this section to mean ``reducible for 3-dynamic 10-paintability''.
We say that a vertex $w$ is \emph{dull}
if $w$ has fewer than $\min\set{3,d_G(w)}-1$ different colors on $N_G(w)$
during a round in the 3-dynamic 10-paintability game.

To show that each of the following configurations is reducible,
by Remark~\ref{rem:extend} it suffices to find a graph $G'$ that is \extend{3} to $G$ and embeds on the torus, 
and a $G'$-first strategy for Painter.
Anytime that our construction $G'$ has edges added to $G$, we ignore multiedges.
For example, in Lemma~\ref{adjacent3and3}, if $y_1z_2,y_2z_2\in E(G)$, then $G'=G-\set{v_1,v_2}$. Additionally, in all cases, the edge can be added to the induced embedding to get a toroidal embedding of $G'$. 

In some cases we must avoid coloring a vertex $v\in V(G)-V(G')$ in a round 
if it has a dull neighbor $u$ such that $uw\in E(G')-E(G)$ and a neighbor of $u$ is being colored by Painter in that round.
The coloring of $V(G')$ that results from the $G'$-first strategy may be $r$-dynamic,
but the color on $w$ doesn't contribute towards $u$ being $r$-dynamic in $G$ since $uw\nin E(G)$.
By ensuring that $v$ adds a color to $N_G(u)$ when $u$ has fewer than $\min\set{3,d_G(w)}-1$ colors,
we prevent $N_G(u)$ from having fewer than $\min\set{d(u),r}$ colors at the end of the game,
even when the color on $w$ is ``lost''.
Note that a vertex $v$ is rejected because a neighbor is dull at most twice for that neighbor.

In the figures for the reducible configurations,
the thick, gray edges represent the edges possibly added by $E'(G)$
and the dashed lines enclose the vertices of $S$.

\begin{lemma}\label{2^-vertex}
A vertex of degree at most $2$ is reducible.
\end{lemma}
\begin{proof}
Suppose $G$ has a $1$-vertex $v$ with neighbor $u$ (Figure~\ref{fig:2^-vertex}$(i)$). 
Let $G'=G-v$.
Observe that $G'$ is \extend{3} to $G$. 
Painter plays a $G'$-first strategy and rejects $v$ on a round
if $u$ is being colored that round according to the $G'$-first strategy.
Painter also rejects $v$ on rounds in which $u$ is dull and a neighbor of $u$ is being colored.
Since $v$ is rejected at most 3 times, Painter wins by Remark~\ref{rem:extend}.

Suppose $G$ has a $2$-vertex $v$ with neighbors $y$ and $z$ (Figure~\ref{fig:2^-vertex}$(ii)$).
Let $G'=(G\cup{yz})-v$. Because $y$ and $z$ are on the same face, we may add the edge $yz$ to $G'$.
Since $yz$ is in $E(G')$, we ensure that when Painter plays a $G'$-first strategy, 
the colors on $y$ and $z$ will be distinct.
Thus $G'$ is \extend{3} to $G$.
Painter rejects $v$ on rounds in which $z$ is colored
or when $z$ is dull and a neighbor of $z$ is being colored.
Symmetrically, Painter rejects $v$ in the corresponding situations with $y$.
Since $v$ is rejected at most $6$ times, Painter wins.
\end{proof}

\tikzstyle{tom label} = [draw=none, fill=none, right]
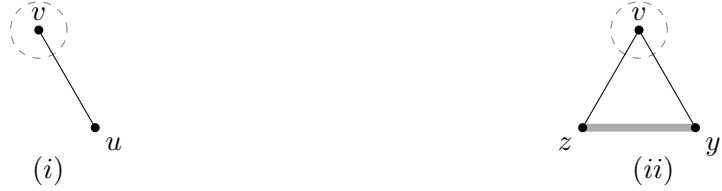
\begin{figure} \centering
\begin{subfigure}[b]{.45\linewidth} \centering
\begin{tikzpicture}[scale=.75]
\path (300:0cm) coordinate (X1); \path (300:2cm) coordinate (X2);
\fill[fill=white,draw=gray,dashed] (X1) circle (.5cm);
\draw (X1) -- (X2);
\fill[draw=black] (X1) circle (2pt);
\fill[draw=black] (X2) circle (2pt);
\node[above] at (X1) {$v$}; 
\node[below right] at (X2) {$u$}; 
\node[tom label] at (-0.3,-2.5) {$(i)$};
\end{tikzpicture}
\end{subfigure}
~~
\begin{subfigure}[b]{.45\linewidth} \centering
\begin{tikzpicture}[scale=.75]
\path (90:0cm) coordinate (X1); \path (300:2cm) coordinate (X2); \path (240:2cm) coordinate (X3);
\fill[fill=white,draw=gray,dashed] (X1) circle (.5cm);
\draw (X1) -- (X2);
\draw (X1) -- (X3);
\draw[Eedge] (X2) -- (X3);
\fill[draw=black] (X1) circle (2pt);
\fill[draw=black] (X2) circle (2pt);
\fill[draw=black] (X3) circle (2pt);
\node[above] at (X1) {$v$}; 
\node[below right] at (X2) {$y$}; 
\node[below left] at (X3) {$z$}; 
\node[tom label] at (-0.3,-2.5) {$(ii)$};
\end{tikzpicture}
\end{subfigure}
\caption{Configuration for Lemma~\ref{2^-vertex}: A $1$-vertex $v$ with neighbor $u$ or a $2$-vertex $v$ with neighbors $y$ and $z$.}
\label{fig:2^-vertex}
\end{figure}

\begin{lemma}\label{adjacent3and3}
A pair of adjacent $3^-$-vertices is reducible.
\end{lemma}
\begin{proof}
Lemma~\ref{2^-vertex} implies that it suffices to consider a pair of vertices $v_1$ and $v_2$
with $d_G(v_i)=3$ for $i\in\set{1,2}$.
Let $y_i$ and $z_i$ be the other neighbors of $v_i$ (Figure~\ref{fig:adjacent3and3}).
Let $G'=(G\cup\set{y_1z_1,y_2z_2})-\set{v_1,v_2}$.
Observe that $G'$ is \extend{3} to $G$.

For $i\in\set{1,2}$, Painter rejects $v_i$ when $y_1,y_2,z_1$, or $z_2$ are colored.
Painter also rejects $v_i$ when $y_i$ or $z_i$ is dull and a neighbor of that vertex is being colored.
Lastly, Painter rejects $v_2$ when $v_1$ is colored.
Since $v_1$ is rejected at most $4+2+2$ times and $v_2$ is rejected at most 9 times, Painter wins.
\end{proof}

\begin{figure} \centering
\begin{tikzpicture}[scale=.75]
\path (0,0) coordinate (V1);
\path (3,0) coordinate (V2); 
\fill[fill=white,draw=gray,dashed] (1.5,0) circle [x radius=2.5cm, y radius=.6cm];
\draw (V1) -- (V2);
\foreach \i in {1,2} {\foreach \j in {1,2} {\path (V\i)++(180+60*\j:2cm) coordinate (Y\j\i);}}
\foreach \i in {1,2} {\draw (V\i) -- (Y1\i); \draw (V\i) -- (Y2\i); 
\draw[Eedge] (Y1\i) -- (Y2\i);
\fill[draw=black] (V\i) circle (2pt); 
\fill[draw=black] (Y1\i) circle (2pt); 
\fill[draw=black] (Y2\i) circle (2pt);}
\node[above] at (V1) {$v_1$}; 
\node[above] at (V2) {$v_2$}; 
\node[below] at (Y11) {$y_1$}; 
\node[below] at (Y12) {$y_2$}; 
\node[below] at (Y21) {$z_1$}; 
\node[below] at (Y22) {$z_2$}; 
\end{tikzpicture}
\caption{Configuration for Lemma~\ref{adjacent3and3}: $d_G(v_1)=d_G(v_2)=3$.}
\label{fig:adjacent3and3}
\end{figure}
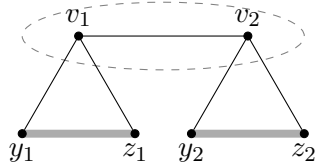

We say that a 3-face is \emph{expensive} when it contains a 3-vertex and
a 4-face is \emph{expensive} when it contains two 3-vertices.
Lemma~\ref{adjacent3and3} implies that
the 3-vertices on an expensive 4-face must be nonadjacent.

\begin{lemma}\label{3nbrsandexpfaces}
Let $v$ be a $d$-vertex that is contained in $e_3$ expensive $3$-faces
and $e_4$ expensive $4$-faces.
If the number of $3$-neighbors of $v$ is $k,$ $k\ge2,$ and $d+k-e_3-e_4<10$, then $v$ is reducible.
\end{lemma}
\begin{proof}
Since $k\ge2$, Lemmas~\ref{2^-vertex} and~\ref{adjacent3and3} imply that $d\ge4$.
Let $\VEC x1k$ be the 3-neighbors of $v$ in clockwise order,
and let $v,y_i,z_i$ be the neighbors in clockwise order of $x_i$ for $i\in[k]$ (Figure~\ref{fig:3nbrsandexpfaces}).
Let $S=\set{v,\VEC x1k}$ and $E'=\set{y_iz_i:i\in[k]}$,
and let $G'=(G\cup E')-S$.
Observe that $G'$ is \extend{3} to $G$.

Painter rejects $x_i$ when $v$, $y_i$, or $z_i$ are colored,
or when $y_i$ or $z_i$ is dull and a neighbor of that vertex is being colored.
Painter also rejects $x_2$ when $x_1$ is colored,
and when $k\ge3$ Painter rejects $x_3$ when $x_2$ or $x_1$ are colored.
Thus each $x_i$ is rejected at most 9 times.
Painter rejects $v$ whenever a vertex of $N_G(v)-\set{x_1,...,x_k}$ is colored.
Painter also rejects $v$ when a vertex in $\set{y_i,z_i:i\in[k]}$ is colored.
Thus Painter rejects $v$ at most $|(N_G(v)-\set{x_1,...,x_k})\cup\set{y_i,z_i:i\in[k]}|$ times. 
When $x_i$ is on an expensive 3-face, then one of $y_i$ or $z_i$ is in $N_G(v)-\set{x_1,...,x_k}$ and is already avoided.
When $x_i$ and $x_{i+1}$ are on an expensive 4-face, we have that $z_i=y_{i+1}$.
Thus $|(N_G(v)-\set{x_1,...,x_k})\cup\set{y_i,z_i:i\in[k]}|\le d-k+2k-e_3-e_4$. If $d-2k-e_3-e_4<10$, then Painter wins.
\end{proof}

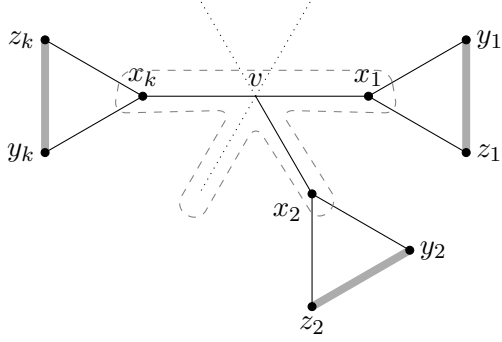
\begin{figure} \centering
\begin{tikzpicture}[scale=.75]
\path (0,0) coordinate (V);
\foreach \j in {1,...,6} {\path (V)++(60-60*\j:2cm) coordinate (X\j);}
\foreach \j in {1,2} {\path (X1)++(90-60*\j:2cm) coordinate (Y\j1);}
\foreach \j in {1,2} {\path (X2)++(30-60*\j:2cm) coordinate (Y\j2);}
\foreach \j in {1,2} {\path (X4)++(270-60*\j:2cm) coordinate (Y\j4);}
\foreach \j in {1,2,4} {
  \draw[Eedge] (Y1\j) -- (Y2\j);
  \foreach \i in {1,2} {
    \draw (X\j) -- (Y\i\j);
    \fill[draw=black] (Y\i\j) circle (2pt);}
  \draw (V) -- (X\j);
  \fill[draw=black] (X\j) circle (2pt); }
\foreach \j in {3,5,6} {\draw[dotted] (V) -- (X\j);}
\node[above] at (V) {$v$}; 
\node[above] at (X1) {$x_1$}; 
\node[below left] at (X2) {$x_2$}; 
\node[above] at (X4) {$x_k$}; 
\node[right] at (Y11) {$y_1$}; 
\node[right] at (Y12) {$y_2$}; 
\node[left] at (Y14) {$y_k$}; 
\node[right] at (Y21) {$z_1$}; 
\node[below] at (Y22) {$z_2$}; 
\node[left] at (Y24) {$z_k$};
\foreach \j in {1,...,3} {\path (V)++(30-60*\j:.5cm) coordinate (X\j5);}
\foreach \j in {1,2} {\path (X1)++(130-80*\j:.6cm) coordinate (X\j1);}
\foreach \j in {1,2} {\path (X2)++(40-60*\j:.5cm) coordinate (X\j2);}
\foreach \j in {1,2} {\path (X3)++(330-60*\j:.5cm) coordinate (X\j3);}
\foreach \j in {1,2} {\path (X4)++(290-80*\j:.6cm) coordinate (X\j4);}
\draw[rounded corners=5pt,gray,dashed] (X11)--(X21)--(X15)--(X12)--(X22)--(X25)--(X13)--(X23)--(X35)--(X14)--(X24)--cycle;
\end{tikzpicture}
\caption{Configuration for Lemma~\ref{3nbrsandexpfaces}: $d_G(v)=d$; $d_G(x_i)=3$ for $i\in[k]$.}
\label{fig:3nbrsandexpfaces}
\end{figure}

\begin{lemma}\label{adjacent3and4}
A $4^-$-vertex with a $3^-$-neighbor is reducible.
\end{lemma}
\begin{proof}
By Lemmas~\ref{adjacent3and3} and \ref{3nbrsandexpfaces}, 
it suffices to consider a 4-vertex $v_1$ with exactly one 3-neighbor $v_2$. 
Let $y_1$ and $z_1$ be neighbors of $v_1$,
and let $y_2$ and $z_2$ be the neighbors of $v_2$ other than $v_1$ (Figure~\ref{fig:adjacent3and4}).
Because $y_1$ and $z_1$ are on the same face, we may add the edge $y_1z_1$ to $G'$.
Let $G'=(G\cup\set{y_1z_1,y_2z_2})-\set{v_1,v_2}$.
Observe that $G'$ is \extend{3} to $G$.

For $i\in\set{1,2}$, Painter rejects $v_i$ when $y_1,z_1,y_2$, or $z_2$ is colored.
Painter also rejects $v_i$ when $y_i$ or $z_i$ is dull and a neighbor of that vertex is being colored.
Additionally, Painter rejects $v_1$ when its neighbor other than $v_2,y_1$, or $z_1$ is colored.
Lastly, Painter rejects $v_2$ when $v_1$ is colored.
For $i\in\set{1,2}$, Painter rejects $v_i$ at most 9 times.
\end{proof}

\begin{figure} \centering
\begin{tikzpicture}[scale=.75]
\path (0,0) coordinate (V1);
\path (3,0) coordinate (V2); 
\fill[fill=white,draw=gray,dashed] (1.5,0) circle [x radius=2.5cm, y radius=.6cm];
\draw (V1) -- (V2);
\foreach \i in {1,2} {\foreach \j in {1,2} {\path (V\i)++(180+60*\j:2cm) coordinate (Y\j\i);}}
\foreach \i in {1,2} {\draw (V\i) -- (Y1\i); \draw (V\i) -- (Y2\i); 
\draw[Eedge] (Y1\i) -- (Y2\i);
\fill[draw=black] (V\i) circle (2pt); 
\fill[draw=black] (Y1\i) circle (2pt); 
\fill[draw=black] (Y2\i) circle (2pt);}
\node[above] at (V1) {$v_1$}; 
\node[above] at (V2) {$v_2$}; 
\node[below] at (Y11) {$y_1$}; 
\node[below] at (Y12) {$y_2$}; 
\node[below] at (Y21) {$z_1$}; 
\node[below] at (Y22) {$z_2$}; 
\end{tikzpicture}
\caption{Configuration for Lemma~\ref{adjacent3and4}: $d_G(v_1)=4$; $d_G(v_2)=3$.}
\label{fig:adjacent3and4}
\end{figure}
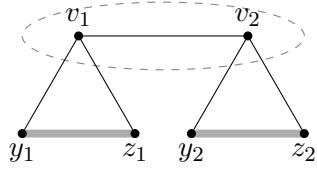

\begin{lemma}\label{4vertex3-cycle}
A (possibly non-facial) 3-cycle containing at most one $5^+$-vertex is reducible.
\end{lemma}
\begin{proof}
If a 3-cycle contains a 3-vertex,
then Lemma~\ref{adjacent3and4} implies the result.
Thus we may assume that the 3-cycle contains a 4-vertex $v_1$.
Lemma~\ref{adjacent3and4}
implies that the other vertices on the 3-cycle are $4^+$-vertices,
so it suffices to consider the case when there is another 4-vertex $v_2$ on the 3-cycle.
Let $z$ be the common neighbor of $v_1$ and $v_2$ on the 3-cycle.
Let $y_i$ be a neighbor of $v_i$ such that $y_i, v_i, z$ are consecutive vertices on a face and $y_i\neq v_{3-i}$, (Figure~\ref{fig:4vertex3-cycle}).
Let $G'=(G\cup\set{y_1z,y_2z})-\set{v_1,v_2}$.
Observe that $G'$ is \extend{3} to $G$.

For $i\in\set{1,2}$, Painter rejects $v_i$ when $y_1,y_2$, or $z$ is colored,
and when $y_i$ is dull and one of its neighbors is being colored.
Also, Painter rejects $v_2$ when $v_1$ is colored.
For $i\in\set{1,2}$, Painter rejects each $v_i$ at most 8 times.
\end{proof}

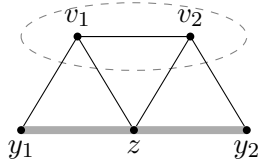
\begin{figure}
\begin{center}
\begin{tikzpicture}[rotate = 180, scale = 1.5]
\path (.5,.83) coordinate (v); 
\fill[fill=white,draw=gray,dashed] (.5,0) circle [x radius=1.cm, y radius=.3cm];
\node[below] at (v) {$z$}; 
\path (-.5,.83) coordinate (a); 
\node[below] at (a) {$y_2$};
\path (1.5,.83) coordinate (z); 
\node[below] at (z) {$y_1$};
\path (0,0) coordinate (y); 
\node[above] at (y) {$v_2$};
\path (1,0) coordinate (x); 
\node[above] at (x) {$v_1$};
\draw (a) -- (v) -- (z) -- (x) -- (y) -- cycle;
\draw (v) -- (y) -- (x) -- cycle; 
\draw[Eedge] (a) -- (z); 
\fill (v) circle (1pt); 
\fill (a) circle (1pt); 
\fill (z) circle (1pt); 
\fill (y) circle (1pt); 
\fill (x) circle (1pt); 
\end{tikzpicture}
\end{center}
\caption{Configuration for Lemma~\ref{4vertex3-cycle}: $d_G(v_1)=d_G(v_2)=4$; $z$ is a common neighbor of $v_1$ and $v_2$.}
\label{fig:4vertex3-cycle}
\end{figure}

\begin{lemma}\label{adjacenttriangle}
Let $uv$ be the common edge of two adjacent 3-faces with $d_G(v) \le 5$. If the neighbors of $v$ off these 3-faces are $4^+$-vertices, then $\{v\}$ is reducible. 
\end{lemma}
\begin{proof}
Suppose that $v$ is a $5^-$-vertex.
Let $uvy$ and $uvz$ be 3-faces (Figure~\ref{fig:adjacenttriangle}).
Note that if $v$ is a $4^-$-vertex, then Lemma~\ref{adjacent3and4} implies any other neighbor of $v$ is a $4^+$-vertex and the additional condition is unnecessary.
Let $G'=(G\cup{yz})-v$.
Observe that $G'$ is \extend{3} to $G$.

Painter rejects $v$ when any vertex in $N_G(v)$ is colored.
Painter also rejects $v$ when $y$ or $z$ is dull and a neighbor of that vertex is being colored.
Since $v$ is rejected at most 9 times, Painter wins.
\end{proof}

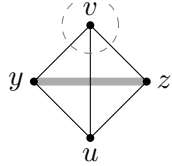
\begin{figure}
\begin{tikzpicture}[scale=.75]
\path (0,-1) coordinate (u); 
\node[below] at (u) {$u$}; 
\path (0,1) coordinate (v); 
\node[above] at (v) {$v$}; 
\path (1,0) coordinate (z); 
\node[right] at (z) {$z$};
\path (-1,0) coordinate (y); 
\node[left] at (y) {$y$};
\draw[Eedge] (y) -- (z); 
\draw (u) -- (y) -- (v) -- (z) -- cycle;
\draw (u) -- (v);
\draw[color = gray, dashed] (v) ellipse (.5cm and .5cm); 
\fill (y) circle (2pt); 
\fill (u) circle (2pt); 
\fill (v) circle (2pt); 
\fill (z) circle (2pt); 
\end{tikzpicture}
\caption{Configuration for Lemma~\ref{adjacenttriangle}: $uv$ is the shared edge; $d_G(v)\le4$.}
\label{fig:adjacenttriangle}
\end{figure}
%
%

\begin{lemma}\label{triangleand4-vtx}
Let $v$ be a $7^-$-vertex with a $4^-$-neighbor $x$. If $v$ has no 3-neighbors (aside from possibly $x$) and is on more than one 3-face, then $\{v,x\}$ is reducible.
\end{lemma}
\begin{proof}
We show that there cannot be a 3-face containing $v$ but not $x$.
After showing this, it follows that any 3-face containing $v$ must use the edge $xv$.
Lemma~\ref{adjacenttriangle} implies there is at most one such face.
Let $y$ and $z$ be other neighbors of $x$ that are on a shared face with $x$ (Figure~\ref{fig:triangleand4-vtx}).
Let $G'=(G\cup{yz})-\set{v,x}$.
Observe that $G'$ is \extend{3} to $G$.

Painter rejects $v$ when any vertex in $N_G(v)\cup\set{y,z}$ is colored.
Painter rejects $x$ when any vertex in $N_G(x)$ is colored,
and when $v$, $y$, or $z$ is dull and a neighbor of the dull vertex is being colored.
Since each of $x$ and $y$ is rejected at most 9 times, Painter wins.
\end{proof}

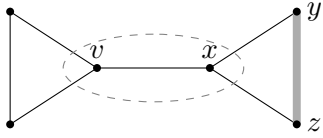
\begin{figure}
\begin{tikzpicture}[rotate = 90, scale = 1.5]
\path (0,-1) coordinate (x); 
\node[above] at (x) {$x$}; 
\path (0,0) coordinate (v); 
\node[above] at (v) {$v$}; 
\path (-.5,-1.77) coordinate (z); 
\node[right] at (z) {$z$};
\path (.5,-1.77) coordinate (y); 
\node[right] at (y) {$y$};
\path (-.5,.77) coordinate (a); 
\path (.5,.77) coordinate (b); 
\draw (v) -- (x); \draw (y) -- (x) -- (z); 
\draw (v) -- (a) -- (b) -- cycle; 
\draw[Eedge] (y) -- (z); 
\draw[color = gray, dashed] (0,-.5) ellipse (.3cm and .8cm); 
\fill (x) circle (1pt); 
\fill (v) circle (1pt); 
\fill (z) circle (1pt); 
\fill (y) circle (1pt); 
\fill (a) circle (1pt); 
\fill (b) circle (1pt); 
\end{tikzpicture}
\caption{Configuration for Lemma~\ref{triangleand4-vtx}: $d_G(v)\le7$; $d_G(x)\le4$; $v$ has no other 3-neighbors.}
\label{fig:triangleand4-vtx}
\end{figure}

\begin{lemma}\label{3triangles}
Vertices $u,v,x,y,z$ forming 3-faces $vzx$, $vxy$, and $vyu$ (vertices given in clockwise order) with $d_G(v)\le 6$ is reducible.
\end{lemma}
\begin{proof}
Suppose $d_G(v)<7$ (Figure~\ref{fig:3triangles}).
Lemma~\ref{adjacenttriangle} implies that $d_G(x)\ge5$ and $d_G(y)\ge5$.
Thus $v$ has at most four $3$-neighbors.
Lemma~\ref{3nbrsandexpfaces} implies that $vyu$ is not an expensive 3-face;
in particular $d_G(u)>3$.
Let $G'=(G\cup{yz})-v$.
Observe that $G'$ is \extend{3} to $G$.

Painter rejects $v$ when any of $N_G(v)$ is colored.
If Painter blindly rejects $v$ when $y$ or $z$ is dull and a neighbor of that vertex is being colored,
then $v$ may be rejected $6+2+2$ times, which is too many so we must state Painter's strategy more carefully.
Since $y$ and $z$ share a common neighbor $x$,
we instead have Painter reject $v$ when any vertex in $N_G(v)\cup\set{x}$ is colored.
Additionally, Painter rejects $v$ if $N_G(y)-\set{v,x}$ has no colors but has a vertex that is being colored;
at the end of the game $N_G(y)-\set{v,x}$ will have at least one color distinct from those on $x$ and $v$, so $y$ will be $3$-dynamic.
Similarly, Painter rejects $v$ if $N_G(z)-\set{v,x}$ has no colors but has a vertex that is being colored.
Under this strategy, $v$ is rejected at most 9 times, and Painter wins.
\end{proof}

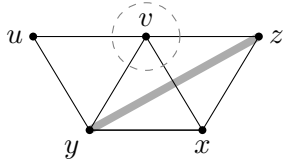
\begin{figure}
\begin{center}
\begin{tikzpicture}[scale = 1.5]
\path (.5,.83) coordinate (v); 
\node[above] at (v) {$v$}; 
\path (-.5,.83) coordinate (a); 
\node[left] at (a) {$u$};
\path (1.5,.83) coordinate (z); 
\node[right] at (z) {$z$};
\path (0,0) coordinate (y); 
\node[below left] at (y) {$y$};
\path (1,0) coordinate (x); 
\node[below] at (x) {$x$};
\draw[Eedge] (y) -- (z); 
\draw (a) -- (v) -- (z) -- (x) -- (y) -- cycle;
\draw (v) -- (y) -- (x) -- cycle; 
\draw[color = gray, dashed] (v) ellipse (.3cm and .3cm); 
\fill (v) circle (1pt); 
\fill (a) circle (1pt); 
\fill (z) circle (1pt); 
\fill (y) circle (1pt); 
\fill (x) circle (1pt); 
\end{tikzpicture}
\end{center}
\caption{Configuration for Lemma~\ref{3triangles}: $d_G(v)\le6$.}
\label{fig:3triangles}
\end{figure}

\begin{lemma}\label{3and4faces}
An expensive 4-face sharing an edge with a 3-face is reducible.
\end{lemma}
\begin{proof}
Lemma~\ref{adjacent3and3} implies that the 3-vertices of the expensive 4-face are not adjacent,
and Lemma~\ref{adjacent3and4} implies that the other vertices on the 4-face are $5^+$-vertices.
Let $vu_1z$ be the 3-face sharing an edge with the 4-face $vyu_2u_1$ where $d_G(v)=d_G(u_2)=3$ (Figure~\ref{fig:3and4faces}).
Let $G'=(G\cup{yz})-v$.
Observe that $G'$ is \extend{3} to $G$, and that $u_2$ guarantees that $u_1$ and $y$ receive distinct colors in any $3$-dynamic coloring of $G'$. 

Painter rejects $v$ when any vertex in $N_G(v)$ is colored,
or when $y$ or $z$ is dull and a neighbor of that vertex is being colored.
Since $v$ is rejected at most 7 times, Painter wins.
\end{proof}

\begin{figure}
\begin{center}
\begin{tikzpicture}[scale = 1.5]
\path (0,0) coordinate (u1); 
\node[below] at (u1) {$u_1$}; 
\path (-.5,.5) coordinate (z); 
\node[left] at (z) {$z$};
\path (0,1) coordinate (v); 
\node[above] at (v) {$v$};
\path (1,1) coordinate (y); 
\node[above] at (y) {$y$};
\path (1,0) coordinate (u2); 
\node[below] at (u2) {$u_2$};
\draw[Eedge] (y) -- (z); 
\draw (u1) -- (z) -- (v) -- (y) -- (u2) -- (u1) -- (v);
\draw[color = gray, dashed] (v) ellipse (.29cm and .29cm); 
\fill (u1) circle (1pt); 
\fill (z) circle (1pt); 
\fill (v) circle (1pt); 
\fill (y) circle (1pt); 
\fill (u2) circle (1pt); 
\end{tikzpicture}
\end{center}
\caption{Configuration for Lemma~\ref{3and4faces}: $d_G(v)=d_G(u_2)=3$.}
\label{fig:3and4faces}
\end{figure}
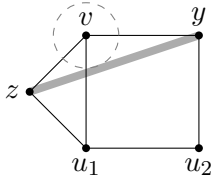

One additional configuration is needed for the torus. 

\begin{lemma}\label{4444face}
A 4-face with no $5^+$-vertex is reducible.
\end{lemma}
\begin{proof}
Lemma~\ref{adjacent3and4} implies that it suffices to consider a face $v_1v_2v_3v_4$ in which all vertices have degree 4.
For $i\in[4]$, let $y_i$ and $z_i$ be neighbors (taken clockwise) of $v_i$ not on the 4-face.
Let $S=\set{\VEC v14}$, and let $G'=G-S$.
Observe that $G'$ is \extend{3} to $G$.

Lemmas~\ref{adjacent3and4} and~\ref{4vertex3-cycle}
imply that $d_G(w)\ge4$ for each $w\in N_G(S)$ and that $z_i\ne y_{i+1}$
where indices are taken mod 4.
Thus if $d_{G'}(y_i)=2$ or $d_{G'}(z_i)=2$ for some $i\in[4]$,
then that vertex is adjacent to $v_{i+2}$ with indices taken mod 4.

Painter rejects $v_1$ in a round if any of the following conditions hold:
\begin{itemize}
\item $y_1$ or $z_1$ is colored.
\item If $d_{G'}(y_1)=2$ and one of the neighbors in $G'$ of $y_1$ is being colored.
\item If $d_{G'}(z_1)=2$ and one of the neighbors in $G'$ of $z_1$ is being colored.
\item Both $y_2$ and $z_2$ are being colored or both $y_4$ and $z_4$ are being colored.
\end{itemize}

Painter rejects $v_2$ in a round if any of the following conditions hold:
\begin{itemize}
\item $v_1$, $y_2$ or $z_2$ is colored.
\item If $d_{G'}(y_2)=2$ and one of the neighbors in $G'$ of $y_2$ is being colored.
\item If $d_{G'}(z_2)=2$ and one of the neighbors in $G'$ of $z_2$ is being colored.
\item Both $y_1$ and $z_1$ are being colored or both $y_3$ and $z_3$ are being colored.
\end{itemize}

Painter rejects $v_3$ in a round if any vertex in $\set{v_1,v_2,    y_2,z_2,y_3,z_3,y_4,z_4}$ is colored.
Painter rejects $v_4$ in a round if any vertex in $\set{v_1,v_2,v_3,y_1,z_1,y_3,z_3,y_4,z_4}$ is colored.
For $i\in[4]$, Painter rejects $v_i$ at most 9 times.

We now verify that this strategy for Painter produces a 3-dynamic coloring of $G$.
Distinct colors are given to $\set{v_1,v_2,v_3,v_4}$, and each $v_i$ avoids the colors on $y_i$ and $z_i$.
Thus Painter's strategy produces a proper coloring of $V(G)$.

If any $w\in N_G(S)$ has $d_{G'}(w)=2$, then its neighbors provide two colors to $N_G(w)$,
and either $v_1$ or $v_2$ will provide a third color.
For any $w\in N_G(S)$ with $d_{G'}(w)\ge3$,
Painter's strategy will result in at least three colors appearing in $N_G(w)-S$.
Since the colors on $\set{v_1,v_2,v_3,v_4}$ are all distinct, it suffices to show that for $i\in[4]$,
either $y_i$ or $z_i$ has a color distinct from both $v_{i-1}$ and $v_{i+1}$ with indices taken mod 4.
The fourth conditions for Painter rejecting $v_1$ or $v_2$,
together with the rules for Painter rejecting $v_3$ and $v_4$, ensures that this is the case.
\end{proof}

\section{discharging for toroidal graphs}\label{sec:discharging}

In this section, we give the discharging argument for Theorem~\ref{thm:maintorus}.

\begin{lemma}\label{lem:discharge}
Every toroidal graph has one of the reducible configurations from Section~\ref{sec:reducibleconfigs}. 
\end{lemma}

\begin{proof}
Suppose Lemma~\ref{lem:discharge} is false and let $G$ be a graph embedded in the torus with none of the configurations from Section~\ref{sec:reducibleconfigs}.

For $x \in V(G) \cup F(G)$, let $c(x)$ be the initial charge on $x$, and let $c^*(x)$ be the final charge on $x$.
We use \emph{face charging}: for a vertex $v$ we set $c(v) = 2d(v)-6$, and for a face $f$ we set $c(f) = \ell(f) - 6$.  
By Euler's formula, the sum of initial charge is
\[ \sum_{v \in V(G)} \!\!\!\left(2d(v) - 6\right) + \!\sum_{f \in F(G)} \!\!\!\left(\ell(f) -6\right) = 4 |E(G)| - 6|V(G)| + 2|E(G)| - 6|F(G)| \le 0. \] 

Using that $G$ has none of the reducible structures from Section~\ref{sec:reducibleconfigs},
we will argue that after following the rules below the final charge on each vertex and face of $G$ is non-negative.
Because charge is only moved, and neither created nor destroyed, 
we use this to obtain a contradiction.
This contradiction shows that every toroidal graph must have a reducible configuration,
and thus that every toroidal graph is 3-dynamically 10-paintable. 

By Lemma~\ref{2^-vertex}, $G$ has no $2^-$-vertices so all vertices start with non-negative charge. Thus our discharging rules consist of vertices giving charge to the faces. Recall that a 3-face is expensive if it has a 3-vertex and a 4-face is expensive if it has two 3-vertices. Say a 3-face is \emph{costly} when it contains a 4-vertex. By Lemma~\ref{4vertex3-cycle} a 3-face has at most one $4^-$-vertex. By Lemma~\ref{adjacent3and3} a 4-face has at most two 3-vertices. Additionally, if a 4-face has two 3-vertices, then Lemma~\ref{adjacent3and4} implies that they are nonadjacent, and since  Lemma~\ref{3nbrsandexpfaces} applied to a $5^-$-vertex $v$ implies that $v$ is not on any expensive faces, we have that the other vertices of an expensive 4-face are $6^+$-vertices. If a 4-face has exactly one 3-vertex, then Lemma~\ref{adjacent3and4} implies that its neighbors are $5^+$-vertices. Finally, Lemma~\ref{adjacent3and3} implies that a 5-face has at most two 3-vertices and that if a 5-face has two 3-vertices, then they are not adjacent. Figure~\ref{fig:rules} shows the discharging rules for all possible faces of length at most 5. 

We move charge according to the following rules which are illustrated in Figure~\ref{fig:rules} for all faces receiving charge. 
\begin{enumerate}
\renewcommand{\theenumi}{R\arabic{enumi}}
\item A 3-face with a 3-vertex takes $\frac{3}{2}$ from each of its $5^+$-vertices. 
\item A 3-face with a 4-vertex takes $\frac{1}{2}$ from each of its 4-vertices and $\frac{5}{4}$ charge from each of its $5^+$-vertices. 
\item A 3-face with no $4^-$-vertices takes $1$ from each of its vertices. 
\item A 4-face with exactly two $3$-vertices takes $1$ from each of its $6^+$-vertices. 
\item A 4-face with exactly one 3-vertex takes $\frac{1}{2}$ from its vertex opposite the 3-vertex and $\frac{3}{4}$ charge from its other two vertices.
\item A 4-face with no 3-vertices takes $\frac{1}{2}$ from each of its vertices. 
\item A 5-face with two 3-vertices takes $\frac{1}{2}$ from their common neighbor on the face and $\frac{1}{4}$ charge from its two other vertices. 
\item A 5-face with at most one 3-vertex takes $\frac{1}{4}$ from each of its $4^+$-vertices. 
\item A $6^+$-face takes $\frac{1}{4}$ from each incident $4^+$-vertex.
\end{enumerate}

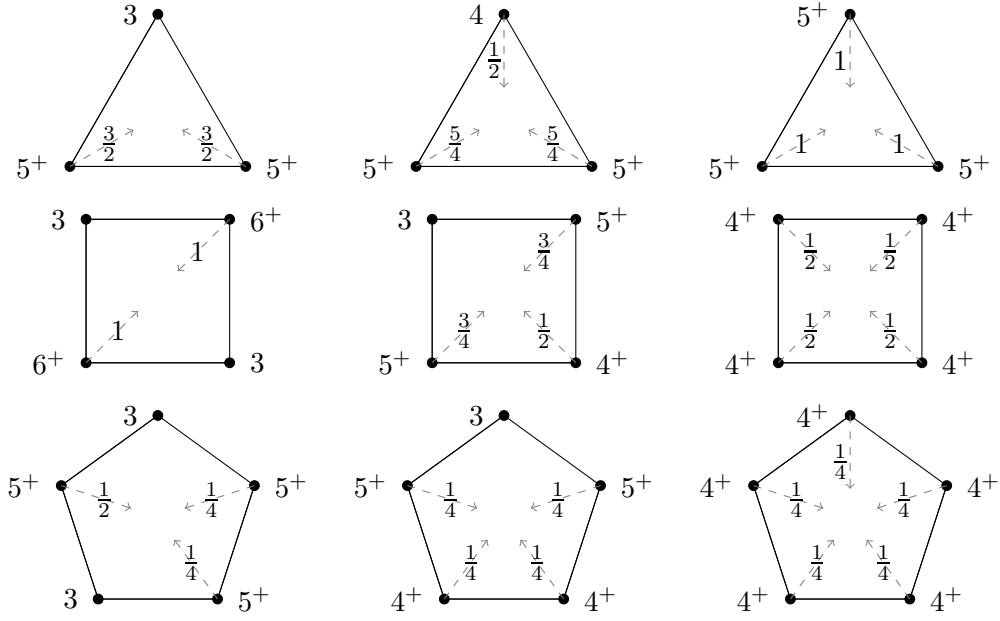
\begin{figure}
\begin{tabular}{ccc}
\begin{tikzpicture}[scale=1.5]
\tikzstyle{vertex}=[circle,fill=black,draw,inner sep=1.3pt]
\draw (90:.9cm) node[vertex] {} -- (210:.9cm) node[vertex] {} -- (330:.9cm) node[vertex] {} -- cycle;
\draw (90:.9cm) node[label=left:{$3$}] {} -- 
	(210:.9cm) node[label=left:{$5^+$}] {} -- 
	(330:.9cm) node[label=right:{$5^+$}] {};
\foreach \a in {210,330} \draw [->,dashed,gray] (\a:.9cm) -- (\a:.25cm);
\draw (210:.5cm) node {${\frac32}$};
\draw (330:.5cm) node {${\frac32}$};
\end{tikzpicture}
&
\begin{tikzpicture}[scale=1.5]
\tikzstyle{vertex}=[circle,fill=black,draw,inner sep=1.3pt]
\draw (90:.9cm) node[vertex] {} -- (210:.9cm) node[vertex] {} -- (330:.9cm) node[vertex] {} -- cycle;
\draw (90:.9cm) node[label=left:{$4$}] {} -- 
	(210:.9cm) node[label=left:{$5^+$}] {} -- 
	(330:.9cm) node[label=right:{$5^+$}] {};
\foreach \a in {90,210,330} \draw [->,dashed,gray] (\a:.9cm) -- (\a:.25cm);
\draw (100:.5cm) node {${\frac12}$};
\draw (210:.5cm) node {${\frac54}$};
\draw (330:.5cm) node {${\frac54}$};
\end{tikzpicture}
&
\begin{tikzpicture}[scale=1.5]
\tikzstyle{vertex}=[circle,fill=black,draw,inner sep=1.3pt]
\draw (90:.9cm) node[vertex] {} -- (210:.9cm) node[vertex] {} -- (330:.9cm) node[vertex] {} -- cycle;
\draw (90:.9cm) node[label=left:{$5^+$}] {} -- 
	(210:.9cm) node[label=left:{$5^+$}] {} -- 
	(330:.9cm) node[label=right:{$5^+$}] {};
\foreach \a in {90,210,330} \draw [->,dashed,gray] (\a:.9cm) -- (\a:.25cm);
\draw (100:.5cm) node {${1}$};
\draw (210:.5cm) node {${1}$};
\draw (330:.5cm) node {${1}$};
\end{tikzpicture}
\\
\begin{tikzpicture}[scale=1.5]
\tikzstyle{vertex}=[circle,fill=black,draw,inner sep=1.3pt]
\draw (45:.9cm) node[vertex] {} -- (135:.9cm) node[vertex] {} -- (225:.9cm) node[vertex] {} -- (315:.9cm) node[vertex] {} -- cycle;
\draw (45:.9cm) node[label=right:{$6^+$}] {} --
	(135:.9cm) node[label=left:{$3$}] {} --
	(225:.9cm) node[label=left:{$6^+$}] {} --
	(315:.9cm) node[label=right:{$3$}] {};
\foreach \a in {45, 225} \draw [->,dashed,gray] (\a:.9cm) -- (\a:.25cm);
\draw (45:.5cm) node {${1}$};
\draw (225:.5cm) node {${1}$};
\end{tikzpicture}
&
\begin{tikzpicture}[scale=1.5]
\tikzstyle{vertex}=[circle,fill=black,draw,inner sep=1.3pt]
\draw (45:.9cm) node[vertex] {} -- (135:.9cm) node[vertex] {} -- (225:.9cm) node[vertex] {} -- (315:.9cm) node[vertex] {} -- cycle;
\draw (45:.9cm) node[label=right:{$5^+$}] {} --
	(135:.9cm) node[label=left:{$3$}] {} --
	(225:.9cm) node[label=left:{$5^+$}] {} --
	(315:.9cm) node[label=right:{$4^+$}] {};
\foreach \a in {45, 225, 315} \draw [->,dashed,gray] (\a:.9cm) -- (\a:.25cm);
\draw (45:.5cm) node {${\frac34}$};
\draw (225:.5cm) node {${\frac34}$};
\draw (315:.5cm) node {${\frac12}$};
\end{tikzpicture}
&
\begin{tikzpicture}[scale=1.5]
\tikzstyle{vertex}=[circle,fill=black,draw,inner sep=1.3pt]
\draw (45:.9cm) node[vertex] {} -- (135:.9cm) node[vertex] {} -- (225:.9cm) node[vertex] {} -- (315:.9cm) node[vertex] {} -- cycle;
\draw (45:.9cm) node[label=right:{$4^+$}] {} --
	(135:.9cm) node[label=left:{$4^+$}] {} --
	(225:.9cm) node[label=left:{$4^+$}] {} --
	(315:.9cm) node[label=right:{$4^+$}] {};
\foreach \a in {45, 135, 225, 315} \draw [->,dashed,gray] (\a:.9cm) -- (\a:.25cm);
\draw (45:.5cm) node {${\frac12}$};
\draw (135:.5cm) node {${\frac12}$};
\draw (225:.5cm) node {${\frac12}$};
\draw (315:.5cm) node {${\frac12}$};
\end{tikzpicture}
\\
\begin{tikzpicture}[scale=1.5]
\tikzstyle{vertex}=[circle,fill=black,draw,inner sep=1.3pt]
\draw (90:.9cm) node[vertex] {} -- (162:.9cm) node[vertex] {} -- (234:.9cm) node[vertex] {} -- (306:.9cm) node[vertex] {} -- (18:.9cm) node[vertex] {} -- cycle;
\draw (90:.9cm) node[label=left:{$3$}] {} -- 
	(162:.9cm) node[label=left:{$5^+$}] {} -- 
	(234:.9cm) node[label=left:{$3$}] {} -- 
	(306:.9cm) node[label=right:{$5^+$}] {} -- 
	(18:.9cm) node[label=right:{$5^+$}] {};
\foreach \a in {18, 162, 306} \draw [->,dashed,gray] (\a:.9cm) -- (\a:.25cm);
\draw (162:.5cm) node {${\frac12}$};
\draw (306:.5cm) node {${\frac14}$};
\draw (18:.5cm) node {${\frac14}$};
\end{tikzpicture}
&
\begin{tikzpicture}[scale=1.5]
\tikzstyle{vertex}=[circle,fill=black,draw,inner sep=1.3pt]
\draw (90:.9cm) node[vertex] {} -- (162:.9cm) node[vertex] {} -- (234:.9cm) node[vertex] {} -- (306:.9cm) node[vertex] {} -- (18:.9cm) node[vertex] {} -- cycle;
\draw (90:.9cm) node[label=left:{$3$}] {} -- 
	(162:.9cm) node[label=left:{$5^+$}] {} -- 
	(234:.9cm) node[label=left:{$4^+$}] {} -- 
	(306:.9cm) node[label=right:{$4^+$}] {} -- 
	(18:.9cm) node[label=right:{$5^+$}] {};
\foreach \a in {18, 162, 234, 306} \draw [->,dashed,gray] (\a:.9cm) -- (\a:.25cm);
\draw (162:.5cm) node {${\frac14}$};
\draw (234:.5cm) node {${\frac14}$};
\draw (306:.5cm) node {${\frac14}$};
\draw (18:.5cm) node {${\frac14}$};
\end{tikzpicture}
&
\begin{tikzpicture}[scale=1.5]
\tikzstyle{vertex}=[circle,fill=black,draw,inner sep=1.3pt]
\draw (90:.9cm) node[vertex] {} -- (162:.9cm) node[vertex] {} -- (234:.9cm) node[vertex] {} -- (306:.9cm) node[vertex] {} -- (18:.9cm) node[vertex] {} -- cycle;
\draw (90:.9cm) node[label=left:{$4^+$}] {} -- 
	(162:.9cm) node[label=left:{$4^+$}] {} -- 
	(234:.9cm) node[label=left:{$4^+$}] {} -- 
	(306:.9cm) node[label=right:{$4^+$}] {} -- 
	(18:.9cm) node[label=right:{$4^+$}] {};
\foreach \a in {18, 90, 162, 234, 306} \draw [->,dashed,gray] (\a:.9cm) -- (\a:.25cm);
\draw (100:.5cm) node {${\frac14}$};
\draw (162:.5cm) node {${\frac14}$};
\draw (234:.5cm) node {${\frac14}$};
\draw (306:.5cm) node {${\frac14}$};
\draw (18:.5cm) node {${\frac14}$};
\end{tikzpicture}
\end{tabular}
\caption{Discharging rules.}
\label{fig:rules}
\end{figure}
 
First we argue that every face ends with non-negative charge. Let $f$ be a face.

Suppose $\ell(f) = 3$, so $c(f) = -3$. Since $f$ has at most one $4^-$-vertex, exactly one of (R1)--(R3) applies to $f$. Under any of (R1)--(R3), $f$ receives 3 so $c^*(f) = 0$. 

Suppose $\ell(f) = 4$, so $c(f) = -2$. The cases that $f$ has two, one, or zero 3-vertices are covered by (R4)--(R6). Under any of (R4)--(R6), $f$ receives 2 so $c^*(f)=0$. 

Suppose $\ell(f) = 5$, so $c(f) = -1$. The cases that $f$ has two 3-vertices, or at most one 3-vertices are covered by (R7) and (R8). Under either, $f$ receives at least 1 and $c^*(f) \ge 0$.

Suppose $\ell(f) \ge 6$, so $c(f) \ge 0$. Since faces never lose charge, and by Lemma~\ref{adjacent3and3} $f$ gains charge from at least one vertex, so we have $c^*(f) > 0$.

\medskip 

It remains to show that every vertex ends with non-negative charge. We further show that $5^+$-vertices end with positive charge. 
Let $w$ be a vertex.
By Lemma~\ref{2^-vertex}, $d(w) \ge 3$.

\textbf{Case 1}: $d(w) = 3$. \\
We have $c(w) = 0$, and $w$ neither gives nor receives charge. 
Thus $c^*(w) = 0$. 

\smallskip 

\textbf{Case 2:} $d(w) = 4$. \\
We have $c(w) = 2$, and  Lemma~\ref{adjacent3and4} implies that $w$ has no 3-neighbors.
Thus $w$ only gives charge under (R2), (R5), (R6), and (R8), and $w$ gives at most $\frac{1}{2}$ to each of four incident faces. Thus $c^*(w) \ge 2 - 4\cdot\frac{1}{2} = 0$.

\smallskip 

\textbf{Case 3:} $d(w) = 5$. \\
We have $c(w) = 4$, and Lemma~\ref{3nbrsandexpfaces} implies that if $w$ has at least two 3-neighbors, then $w$ has exactly five 3-neighbors.
Thus $w$ can have zero, one, or five 3-neighbors. We break into cases based on the number of 3-neighbors of $w$. 

\emph{Case 3a:} $w$ has no 3-neighbors. \\
By (R2) and (R3), $w$ gives at most $\frac{5}{4}$ to costly 3-faces and at most 1 to other 3-faces; 
by (R5), (R6), and (R8) $w$ gives at most $\frac{1}{2}$ to each incident $4^+$-face. 
By Lemma~\ref{adjacenttriangle}, $w$ is on at most two 3-faces. 
If $w$ is not on any costly 3-face, then $c^*(w)\geq4-2\cdot1-3\cdot\frac{1}{2}>0$. 
If $w$ is on a costly 3-face, then by Lemma~\ref{triangleand4-vtx} it is on no other 3-face and so
 $c^*(w)\geq4-1\cdot\frac{5}{4}-4\cdot\frac{1}{2}>0$.

\emph{Case 3b:} $w$ has exactly one 3-neighbor $x$. \\
Lemma~\ref{triangleand4-vtx} implies that $w$ is incident to at most one 3-face as in Figure~\ref{fig:3b4b}.
If $w$ is not on a 3-face, then by (R5)--(R8), it gives at most $\frac{3}{4}$ to each incident face.
Thus $c^*(w) \ge 4 - 5 \cdot \frac{3}{4} > 0$.

If $w$ is on a 3-face $f$ with $x$, then (R1) implies that $w$ gives $\frac{3}{2}$ to $f$.
By (R5)--(R8), $w$ gives at most $\frac{3}{4}$ to the other face shared with $x$, and at most $\frac{1}{2}$ to each other incident face.
Thus $c^*(w) \geq 4 - 1 \cdot \frac{3}{2} - 1 \cdot \frac{3}{4} - 3 \cdot \frac{1}{2} >0$.
If $w$ is on a 3-face $f$ that does not contain $x$, then (R5)--(R8) implies that $w$ gives at most $\frac{3}{4}$ to each face shared with $x$, at most $\frac{5}{4}$ to $f$, and at most $\frac{1}{2}$ to the remaining incident faces.
Thus $c^*(w)\geq 4-2\cdot\frac{3}{4} - 1\cdot \frac{5}{4} - 2\cdot \frac{1}{2} > 0$.

\emph{Case 3c:} $w$ has five 3-neighbors. \\
Lemma~\ref{3nbrsandexpfaces} implies $w$ is not on any $4^-$-faces because such a face would be expensive.
So by (R7), and (R8) $w$ gives at most $\frac{1}{2}$ to each face, and thus $c^*(w)\geq 4-5 \cdot \frac{1}{2} >0 $.

\smallskip

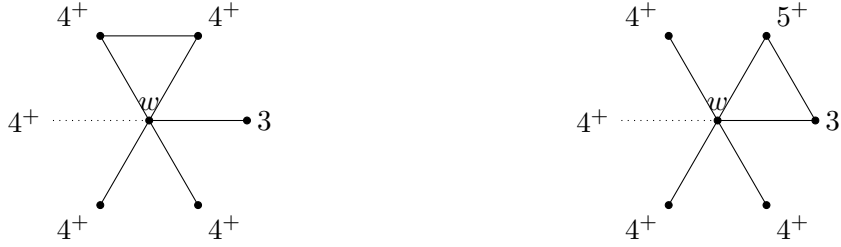
\begin{figure} \centering
\begin{subfigure}[b]{.45\linewidth} \centering
\begin{tikzpicture}[scale=.65]
\path (0,0) coordinate (V);
\foreach \j in {1,...,6} {\path (V)++(60-60*\j:2cm) coordinate (X\j);}
\foreach \j in {1,...,3,5,6} {
  \draw (V) -- (X\j);
  \fill[draw=black] (X\j) circle (2pt); }
\draw[dotted] (V) -- (X4);
\draw (X5) -- (X6);
\fill[draw=black] (V) circle (2pt);
\node[above] at (V) {$w$}; 
\node[right] at (X1) {$3$}; 
\node[below right] at (X2) {$4^+$}; 
\node[below left] at (X3) {$4^+$}; 
\node[left] at (X4) {$4^+$}; 
\node[above left] at (X5) {$4^+$}; 
\node[above right] at (X6) {$4^+$}; 
\end{tikzpicture}
\end{subfigure}
~
\begin{subfigure}[b]{.45\linewidth} \centering
\begin{tikzpicture}[scale=.65]
\path (0,0) coordinate (V);
\foreach \j in {1,...,6} {\path (V)++(60-60*\j:2cm) coordinate (X\j);}
\foreach \j in {1,...,3,5,6} {
  \draw (V) -- (X\j);
  \fill[draw=black] (X\j) circle (2pt); }
\draw[dotted] (V) -- (X4);
\draw (X1) -- (X6);
\fill[draw=black] (V) circle (2pt); 
\node[above] at (V) {$w$}; 
\node[right] at (X1) {$3$}; 
\node[below right] at (X2) {$4^+$}; 
\node[below left] at (X3) {$4^+$}; 
\node[left] at (X4) {$4^+$}; 
\node[above left] at (X5) {$4^+$}; 
\node[above right] at (X6) {$5^+$}; 
\end{tikzpicture}
\end{subfigure}
\caption{Cases 3b and 4b: $d(w)\in\{5,6\}$; $w$ has exactly one 3-neighbor.}
\label{fig:3b4b}
\end{figure}

\textbf{Case 4:} $d(w) = 6$. \\
We have $c(w) = 6$. 
Lemma~\ref{3nbrsandexpfaces} implies that $w$ does not have exactly two or three 3-neighbors. 
We break into cases based on the number of 3-neighbors of $w$.

\emph{Case 4a:} $w$ has no 3-neighbors. \\
Lemma~\ref{3triangles} implies that $w$ is on at most four 3-faces. By (R3), $w$ gives 1 to each incident 3-face and by (R5), (R6), and (R8) $w$ gives at most $\frac{1}{2}$ to each other incident face.
Thus $c^*(w)\geq 6 - 4 \cdot 1 - 2 \cdot \frac{1}{2} >0$.

\emph{Case 4b:} $w$ has exactly one 3-neighbor $x$. \\
Lemma~\ref{triangleand4-vtx} implies $w$ is incident to at most one 3-face as in Figure~\ref{fig:3b4b}. If $w$ has no incident 3-faces, then by (R5)--(R8) $w$ gives at most $\frac{3}{4}$ to the faces with $x$ and at most $\frac{1}{2}$ to each other incident face. Thus $c^*(w) \ge 6 - 2 \cdot \frac{3}{4} - 4 \cdot \frac{1}{2} >0$. If $w$ is on a 3-face $f$ containing $x$, then by (R1) $w$ gives $\frac{3}{2}$ to $f$, and by (R5), (R7), and (R8) $w$ gives at most $\frac{3}{4}$ to the other incident face containing $x$, and at most $\frac{1}{2}$ to each other incident face. So 
$c^*(w)\geq 6 - 1 \cdot \frac{3}{2} - 1 \cdot \frac{3}{4} - 4 \cdot \frac{1}{2} >0$.  
Finally, if $w$ is on a 3-face $f$ that does not contain $x$, then by (R5)--(R8) $w$ gives at most $\frac{3}{4}$ to each incident face containing $x$.
By (R2), and (R3), $w$ gives at most $\frac{5}{4}$ to $f$, and by (R5), (R6), and (R8) $w$ gives at most $\frac{1}{2}$ to the remaining incident faces. Thus  
$c^*(w) \geq 6 - 2 \cdot \frac{3}{4} - 1 \cdot \frac{5}{4} - 3 \cdot \frac{1}{2} >0$.

\emph{Case 4c:} $w$ has exactly four 3-neighbors. \\
Lemma~\ref{3nbrsandexpfaces} implies that $w$ is on no expensive faces.
So $w$ is on at most one 3-face, to which $w$ gives at most $\frac{5}{4}$.
Furthermore, $w$ gives at most $\frac{3}{4}$ to each other incident face.
Thus $c^*(w) \geq 6 - 1 \cdot \frac{5}{4} - 5 \cdot \frac{3}{4} >0$.

\emph{Case 4d:} $w$ has exactly five 3-neighbors. \\
Lemma~\ref{3nbrsandexpfaces} implies that $w$ is on at most one expensive face $f$.
So $w$ gives at most $\frac{3}{2}$ to $f$, if it exists, and at most $\frac{3}{4}$ to any other incident face.
Thus $c^*(v) \geq 6 - 1 \cdot \frac{3}{2} - 5 \cdot \frac{3}{4} >0$.

\emph{Case 4e:} $w$ has six 3-neighbors. \\
Lemma~\ref{3nbrsandexpfaces} implies that $w$ is on at most two expensive faces $f$ and $f'$.
Lemma~\ref{adjacent3and3} implies that these are 4-faces.
So $w$ gives at most 1 to $f$, at most 1 to $f'$, and at most $\frac{3}{4}$ to any other incident face.
Thus $c^*(v)\geq 6- 2 \cdot 1 - 4 \cdot \frac{3}{4} >0$.

\smallskip

\textbf{Case 5:} $d(w) = 7$. \\
We have $c(w) = 8$, and Lemma~\ref{3nbrsandexpfaces} implies that $w$ does not have exactly two 3-neighbors.
We break into cases based on the number of 3-neighbors of $w$. 

\emph{Case 5a:} $w$ has no 3-neighbors. \\
By (R2), $w$ gives at most $\frac{5}{4}$ to a costly 3-face, and $w$ is incident to at most 4 such faces. By (R3),and (R5)--(R8), $w$ gives at most $1$ to each other incident face. Additionally, if $w$ is on four costly 3-faces  (Figure~\ref{fig:T5a}), then Lemma~\ref{4vertex3-cycle} implies that at least one of the other faces containing $w$ is not a 3-face 
and thus by (R5)--(R9) takes at most $\frac{1}{2}$ from $w$.
Thus either $c^*(w)\geq 8 - 3 \cdot \frac{5}{4} - 4 \cdot 1 > 0$ or $c^*(w)\ge 8 - 4 \cdot \frac{5}{4} - 2 \cdot 1 - 1 \cdot \frac{1}{2} >0$ based on whether or not $w$ is on four costly 3-faces. 

\tikzstyle{tom label} = [draw=none, fill=none, right]
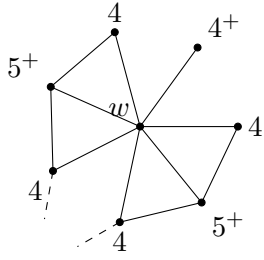
\begin{figure} \centering
\begin{tikzpicture}[scale=.65]
\path (0,0) coordinate (V);
\foreach \j in {1,...,7} {\path (V)++(51-51*\j:2cm) coordinate (X\j);}
\path (X3)++(210:1cm) coordinate (X31);
\path (X4)++(260:1cm) coordinate (X41);
\foreach \j in {1,...,7} {
  \draw (V) -- (X\j);
  \fill[draw=black] (X\j) circle (2pt); }
\draw (X1) -- (X2);
\draw (X2) -- (X3);
\draw (X4) -- (X5);
\draw (X5) -- (X6);
\draw[dashed] (X3) -- (X31);
\draw[dashed] (X4) -- (X41);
\fill[draw=black] (V) circle (2pt); 
\node[above left] at (V) {$w$}; 
\node[right] at (X1) {$4$}; 
\node[below right] at (X2) {$5^+$}; 
\node[below] at (X3) {$4$}; 
\node[below left] at (X4) {$4$}; 
\node[above left] at (X5) {$5^+$}; 
\node[above] at (X6) {$4$}; 
\node[above right] at (X7) {$4^+$}; 
\end{tikzpicture}
\caption{Case 5a: $d(w)=7$; $w$ has no 3-neighbors; subcase where $w$ is on four costly 3-faces.}
\label{fig:T5a}
\end{figure}

\emph{Case 5b:} $w$ has exactly one 3-neighbor $x$. \\
Lemma~\ref{triangleand4-vtx} implies that at most one 3-face contains $w$ and $x$;
if such a face exists (Figure~\ref{fig:5b}$(i)$), then $w$ gives $\frac{3}{2}$ to it by (R1).
By (R5)--(R8), $w$ gives at most $\frac{3}{4}$ to any other face incident to $x$ and $w$.
Lemma~\ref{adjacenttriangle} implies that $w$ is incident to at most four costly 3-faces.
By (R2), $w$ gives $\frac{5}{4}$ to each costly 3-face.
Finally, by (R3), and (R6)--(R8) $w$ gives at most $1$ to each other incident face.
If $w$ is on at most three costly 3-faces, then $c^*(w) \geq 8 - 1 \cdot \frac{3}{2} - 1 \cdot \frac{3}{4} - 2 \cdot \frac{5}{4} - 2 \cdot 1 = 0$.
Otherwise, $w$ is on four costly 3-faces (Figure~\ref{fig:5b}$(ii)$), then neither of the faces containing $x$ are 3-faces.
Thus $c^*(w) \geq 8 - 2 \cdot \frac{3}{4} - 4 \cdot \frac{5}{4} - 1 \cdot 1 > 0$.

To show that $c^*(w) > 0$, we must only consider the case that $x$ and $w$ share an expensive 3-face and a 4-face, and that $w$ is on exactly three costly 3-faces;
possible instances of this configuration are shown in Figure~\ref{fig:T5b}.
However, Lemma~\ref{4vertex3-cycle} implies that if $w$ is contained in three costly 3-faces, then $w$ is contained in a $4^+$-face, which takes at most $\frac{1}{2}$ from $w$.
Thus $c^*(w)\ge\frac{1}{2}$.

\tikzstyle{tom label} = [draw=none, fill=none, right]
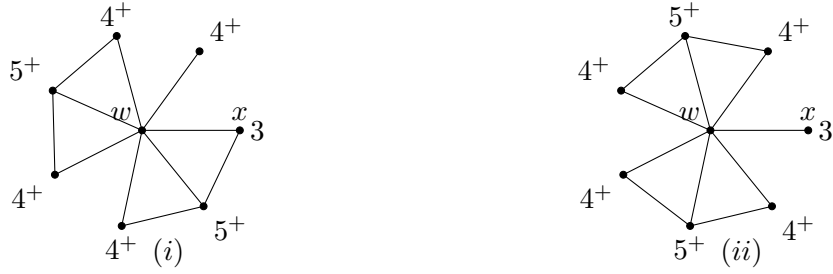
\begin{figure} \centering
\begin{subfigure}[b]{.45\linewidth} \centering
\begin{tikzpicture}[scale=.65]
\path (0,0) coordinate (V);
\foreach \j in {1,...,7} {\path (V)++(51-51*\j:2cm) coordinate (X\j);}
\foreach \j in {1,...,7} {
  \draw (V) -- (X\j);
  \fill[draw=black] (X\j) circle (2pt); }
\draw (X1) -- (X2);
\draw (X2) -- (X3);
\draw (X4) -- (X5);
\draw (X5) -- (X6);
\fill[draw=black] (V) circle (2pt); 
\node[above left] at (V) {$w$}; 
\node[right] at (X1) {$3$}; 
\node[above] at (X1) {$x$}; 
\node[below right] at (X2) {$5^+$}; 
\node[below] at (X3) {$4^+$}; 
\node[below left] at (X4) {$4^+$}; 
\node[above left] at (X5) {$5^+$}; 
\node[above] at (X6) {$4^+$}; 
\node[above right] at (X7) {$4^+$}; 
\node[tom label] at (0,-2.5) {$(i)$};
\end{tikzpicture}
\end{subfigure}
~
\begin{subfigure}[b]{.45\linewidth} \centering
\begin{tikzpicture}[scale=.65]
\path (0,0) coordinate (V);
\foreach \j in {1,...,7} {\path (V)++(51-51*\j:2cm) coordinate (X\j);}
\foreach \j in {1,...,7} {
  \draw (V) -- (X\j);
  \fill[draw=black] (X\j) circle (2pt); }
\draw (X2) -- (X3);
\draw (X3) -- (X4);
\draw (X5) -- (X6);
\draw (X6) -- (X7);
\fill[draw=black] (V) circle (2pt); 
\node[above left] at (V) {$w$}; 
\node[right] at (X1) {$3$}; 
\node[above] at (X1) {$x$}; 
\node[below right] at (X2) {$4^+$}; 
\node[below] at (X3) {$5^+$}; 
\node[below left] at (X4) {$4^+$}; 
\node[above left] at (X5) {$4^+$}; 
\node[above] at (X6) {$5^+$}; 
\node[above right] at (X7) {$4^+$}; 
\node[tom label] at (0,-2.5) {$(ii)$};
\end{tikzpicture}
\end{subfigure}
\caption{Case 5b: $d(w)=7$; $w$ has exactly one 3-neighbor $x$.}
\label{fig:5b}
\end{figure}

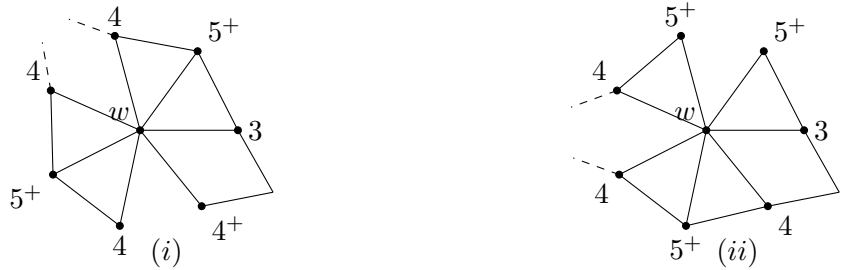
\begin{figure} \centering
\begin{subfigure}[b]{.45\linewidth} \centering
\begin{tikzpicture}[scale=.65]
\path (0,0) coordinate (V);
\foreach \j in {1,...,7} {\path (V)++(51-51*\j:2cm) coordinate (X\j);}
\path (V)++(-25:3cm) coordinate (X11);
\foreach \j in {1,...,7} {
  \draw (V) -- (X\j);
  \fill[draw=black] (X\j) circle (2pt); }
\draw (X1) -- (X11);
\draw (X11) -- (X2);
\draw (X3) -- (X4);
\draw (X4) -- (X5);
\draw (X6) -- (X7);
\draw (X7) -- (X1);
\path (X5)++(100:1cm) coordinate (X51);
\path (X6)++(160:1cm) coordinate (X61);
\draw[dashed] (X5) -- (X51);
\draw[dashed] (X6) -- (X61);
\fill[draw=black] (V) circle (2pt); 
\node[above left] at (V) {$w$}; 
\node[right] at (X1) {$3$}; 
\node[below right] at (X2) {$4^+$}; 
\node[below] at (X3) {$4$}; 
\node[below left] at (X4) {$5^+$}; 
\node[above left] at (X5) {$4$}; 
\node[above] at (X6) {$4$}; 
\node[above right] at (X7) {$5^+$}; 
\node[tom label] at (0,-2.5) {$(i)$};
\end{tikzpicture}
\end{subfigure}
~
\begin{subfigure}[b]{.45\linewidth} \centering
\begin{tikzpicture}[scale=.65]
\path (0,0) coordinate (V);
\foreach \j in {1,...,7} {\path (V)++(51-51*\j:2cm) coordinate (X\j);}
\path (V)++(-25:3cm) coordinate (X11);
\foreach \j in {1,...,7} {
  \draw (V) -- (X\j);
  \fill[draw=black] (X\j) circle (2pt); }
\draw (X1) -- (X11);
\draw (X11) -- (X2);
\draw (X2) -- (X3);
\draw (X3) -- (X4);
\draw (X5) -- (X6);
\draw (X7) -- (X1);
\path (X4)++(160:1cm) coordinate (X41);
\path (X5)++(200:1cm) coordinate (X51);
\draw[dashed] (X4) -- (X41);
\draw[dashed] (X5) -- (X51);
\fill[draw=black] (V) circle (2pt); 
\node[above left] at (V) {$w$}; 
\node[right] at (X1) {$3$}; 
\node[below right] at (X2) {$4$}; 
\node[below] at (X3) {$5^+$}; 
\node[below left] at (X4) {$4$}; 
\node[above left] at (X5) {$4$}; 
\node[above] at (X6) {$5^+$}; 
\node[above right] at (X7) {$5^+$}; 
\node[tom label] at (0,-2.5) {$(ii)$};
\end{tikzpicture}
\end{subfigure}
\caption{Case 5b continued: $d(w)=7$; $w$ has exactly one 3-neighbor.}
\label{fig:T5b}
\end{figure}

\emph{Case 5c:} $w$ has exactly three 3-neighbors (Figure~\ref{fig:5c}). \\
Lemma~\ref{3nbrsandexpfaces} implies that $w$ is not on any expensive faces.
Also, $w$ is on at most two costly 3-faces.
By (R2), $w$ gives $\frac{5}{4}$ to any costly 3-face and by (R3), and (R5)--(R8) $w$ gives at most 1 to each other incident face.
Thus $c^*(w) \geq 8 - 2 \cdot \frac{5}{4} - 5 \cdot 1 >0$.

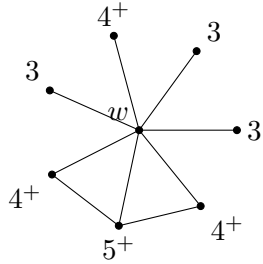
\begin{figure} \centering
\begin{tikzpicture}[scale=.65]
\path (0,0) coordinate (V);
\foreach \j in {1,...,7} {\path (V)++(51-51*\j:2cm) coordinate (X\j);}
\foreach \j in {1,...,7} {
  \draw (V) -- (X\j);
  \fill[draw=black] (X\j) circle (2pt); }
\draw (X2) -- (X3);
\draw (X3) -- (X4);
\fill[draw=black] (V) circle (2pt); 
\node[above left] at (V) {$w$}; 
\node[right] at (X1) {$3$}; 
\node[below right] at (X2) {$4^+$}; 
\node[below] at (X3) {$5^+$}; 
\node[below left] at (X4) {$4^+$}; 
\node[above left] at (X5) {$3$}; 
\node[above] at (X6) {$4^+$}; 
\node[above right] at (X7) {$3$}; 
\end{tikzpicture}
\caption{Case 5c: $d(w)=7$; $w$ has exactly three 3-neighbors.}
\label{fig:5c}
\end{figure}

\emph{Case 5d:} $w$ has exactly four 3-neighbors. \\
Lemma~\ref{3nbrsandexpfaces} implies that $w$ is on at most one expensive face. Additionally, $w$ is on at most two costly 3-faces. So $c^*(w) \geq 8 - 1 \cdot \frac{3}{2} - 2 \cdot \frac{5}{4} - 4 \cdot 1 = 0$. 

To show $c^*(w) >0$, we must only consider the case that $w$ is on one expensive 3-face, no expensive 4-faces, and two costly 3-faces. However, if $w$ is contained in two costly 3-faces, then the four 3-neighbors of $w$ are consecutive in the cyclic order around $w$.
Thus $w$ does not have an expensive 3-face and $c^*(w) > 0$. 

\emph{Case 5e:} $w$ has exactly five 3-neighbors. \\
Lemma~\ref{3nbrsandexpfaces} implies that $w$ is on at most two expensive faces.
By considering the cyclic ordering of the 3-neighbors and the number of expensive faces, we have that $w$ is on at most three 3-faces and if $w$ is on three 3-faces, then at least one is neither expensive nor costly.
If $w$ is on two expensive 3-faces, then $c^*(w)\geq 8 - 2 \cdot \frac{3}{2} - 1 - 4 \cdot \frac{3}{4} >0$.
If $w$ is on at most one expensive 3-face, then $c^*(w) \geq 8 - 1 \cdot \frac{3}{2} - 1 \cdot \frac{5}{4} - 5 \cdot 1 >0$.

\emph{Case 5f:} $w$ has exactly six 3-neighbors. \\
Lemma~\ref{3nbrsandexpfaces} implies that $w$ is on at most three expensive faces and at most one 3-face.
Thus $c^*(w) \geq 8 - 3 \cdot \frac{3}{2} - 1 \cdot 1 - 3 \cdot \frac{3}{4} >0$.

\emph{Case 5g:} $w$ has seven 3-neighbors. \\
Lemma~\ref{adjacent3and3} implies that $w$ is on no 3-faces, so by (R4), and (R7) $c^*(w) \geq 8-7 \cdot 1 = 1$.

\smallskip

\textbf{Case 6:} $d(w) = 8$. \\
We have $c(w) = 10$, and we break into cases based on the configurations of 3-faces around $w$. 

Consider the faces in cyclic order around $w$.
A \emph{run} of 3-faces around $w$ is a set of 3-faces $f_1,\ldots,f_k$ such that for $i\in[k-1]$, faces $f_i$ and $f_{i+1}$ share an edge of the form $wu$ for some vertex $u$.
We consider the maximal runs of 3-faces.
If there are at least three consecutive 3-faces around $w$, then Lemma~\ref{adjacenttriangle} implies that the 3-faces that have two adjacent 3-faces around $w$ have only $5^+$-vertices.
By (R3), $w$ gives at most 1 to these 3-faces, and by (R1)--(R3) $w$ gives at most $\frac{3}{2}$ to other 3-faces.
Lastly, by (R4)--(R8) $w$ gives at most $1$ to each incident $4^+$-face.

\emph{Case 6a:} $w$ is on at most four 3-faces. \\
By (R1)--(R3), $w$ gives at most $\frac{3}{2}$ to each of them, and by (R4)--(R8) $w$ gives at most 1 to each other incident face.
Thus $c^*(w)\geq 10-4\cdot\frac{3}{2}-4\cdot1=0$.
To show $c^*(w) > 0$, we must only consider the case that $w$ is on exactly four 3-faces, all of which are expensive; possible instances of this configuration are shown in Figure~\ref{fig:T6a}.
Lemmas~\ref{adjacenttriangle} and~\ref{3and4faces} imply that at least one of the $4^+$-faces containing $w$ and a 3-neighbor of $w$ takes at most $\frac{3}{4}$ from $w$.
Thus $c^*(w)\ge10 - 4 \cdot \frac{3}{2} - 1 \cdot \frac{3}{4} - 3 \cdot 1 >0$.

\begin{figure} \centering
\begin{subfigure}[b]{.45\linewidth} \centering
\begin{tikzpicture}[scale=.65]
\path (0,0) coordinate (V);
\path (V)++(70:.5cm) coordinate (V1);
\foreach \j in {1,...,8} {\path (V)++(45-45*\j:2cm) coordinate (X\j);}
\foreach \j in {1,...,8} {
  \draw (V) -- (X\j);
  \fill[draw=black] (X\j) circle (2pt); }
\draw (X1) -- (X2);
\draw (X2) -- (X3);
\draw (X6) -- (X7);
\draw (X7) -- (X8);
\path (X5)++(140:1cm) coordinate (X51);
\path (X6)++(180:1cm) coordinate (X61);
\draw[dashed] (X5) -- (X51);
\draw[dashed] (X6) -- (X61);
\path (X8)++(360:1cm) coordinate (X81);
\path (X1)++(40:1cm) coordinate (X11);
\draw[dashed] (X8) -- (X81);
\draw[dashed] (X1) -- (X11);
\fill[draw=black] (V) circle (2pt); 
\node at (V1) {$w$}; 
\node[right] at (X1) {$3$}; 
\node[below right] at (X2) {$5^+$}; 
\node[below] at (X3) {$3$}; 
\node[below left] at (X4) {$3^+$}; 
\node[left] at (X5) {$3^+$}; 
\node[above left] at (X6) {$3$}; 
\node[above] at (X7) {$5^+$}; 
\node[above right] at (X8) {$3$}; 
\end{tikzpicture}
\end{subfigure}
~
\begin{subfigure}[b]{.45\linewidth} \centering
\begin{tikzpicture}[scale=.65]
\path (0,0) coordinate (V);
\path (V)++(70:.5cm) coordinate (V1);
\foreach \j in {1,...,8} {\path (V)++(45-45*\j:2cm) coordinate (X\j);}
\foreach \j in {1,...,8} {
  \draw (V) -- (X\j);
  \fill[draw=black] (X\j) circle (2pt); }
\draw (X1) -- (X2);
\draw (X3) -- (X4);
\draw (X5) -- (X6);
\draw (X7) -- (X8);
\path (X4)++(180:1cm) coordinate (X41);
\path (X5)++(220:1cm) coordinate (X51);
\draw[dashed] (X4) -- (X41);
\draw[dashed] (X5) -- (X51);
\path (X8)++(360:1cm) coordinate (X81);
\path (X1)++(40:1cm) coordinate (X11);
\draw[dashed] (X8) -- (X81);
\draw[dashed] (X1) -- (X11);
\fill[draw=black] (V) circle (2pt); 
\node at (V1) {$w$}; 
\node[right] at (X1) {$3$}; 
\node[below right] at (X2) {$5^+$}; 
\node[below] at (X3) {$5^+$}; 
\node[below left] at (X4) {$3$}; 
\node[left] at (X5) {$3$}; 
\node[above left] at (X6) {$5^+$}; 
\node[above] at (X7) {$3$}; 
\node[above right] at (X8) {$5^+$}; 
\end{tikzpicture}
\end{subfigure}
\caption{Case 6a: $d(w)=8$; subcase where $w$ is on exactly four 3-faces, all of which are expensive.}
\label{fig:T6a}
\end{figure}
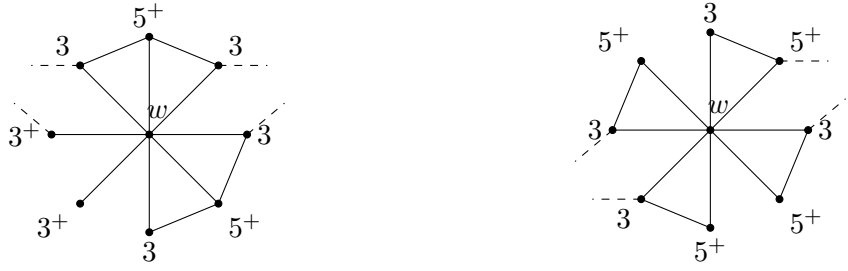

\emph{Case 6b:} $w$ has at most two maximal runs of 3-faces. \\
Thus $c^*(w)\geq 10 - 4 \cdot \frac{3}{2} - 4 \cdot 1 = 0$. Note that if $w$ is on at least six 3-faces, then we are in this case.

To show $c^*(w) > 0$, we must only consider the case that $w$ is on at least five 3-faces, there are two maximal runs of 3-faces, and the first and last 3-faces on the runs are expensive;
possible instances of this configuration are shown in Figure~\ref{fig:T6b}.
Lemma~\ref{3and4faces} implies that the runs end with $4^+$-faces that take at most $\frac{3}{4}$ from $w$ under (R5).
Thus $c^*(w)\ge10 - 4 \cdot \frac{3}{2} - 1 \cdot \frac{3}{4} - 3 \cdot 1 > 0
$.
\begin{figure} \centering
\begin{subfigure}[b]{.45\linewidth} \centering
\begin{tikzpicture}[scale=.65]
\path (0,0) coordinate (V);
\path (V)++(70:.5cm) coordinate (V1);
\foreach \j in {1,...,8} {\path (V)++(45-45*\j:2cm) coordinate (X\j);}
\foreach \j in {1,...,8} {
  \draw (V) -- (X\j);
  \fill[draw=black] (X\j) circle (2pt); }
\draw (X1) -- (X2);
\draw (X2) -- (X3);
\draw (X3) -- (X4);
\draw (X5) -- (X6);
\draw (X6) -- (X7);
\path (X4)++(180:1cm) coordinate (X41);
\path (X5)++(220:1cm) coordinate (X51);
\draw[dashed] (X4) -- (X41);
\draw[dashed] (X5) -- (X51);
\path (X8)++(360:1cm) coordinate (X81);
\path (X1)++(40:1cm) coordinate (X11);
\draw[dashed] (X8) -- (X81);
\draw[dashed] (X1) -- (X11);
\fill[draw=black] (V) circle (2pt); 
\node at (V1) {$w$}; 
\node[right] at (X1) {$3$}; 
\node[below right] at (X2) {$5^+$}; 
\node[below] at (X3) {$5^+$}; 
\node[below left] at (X4) {$3$}; 
\node[left] at (X5) {$3$}; 
\node[above left] at (X6) {$5^+$}; 
\node[above] at (X7) {$3$}; 
\node[above right] at (X8) {$3^+$}; 
\end{tikzpicture}
\end{subfigure}
~
\begin{subfigure}[b]{.45\linewidth} \centering
\begin{tikzpicture}[scale=.65]
\path (0,0) coordinate (V);
\path (V)++(70:.5cm) coordinate (V1);
\foreach \j in {1,...,8} {\path (V)++(45-45*\j:2cm) coordinate (X\j);}
\foreach \j in {1,...,8} {
  \draw (V) -- (X\j);
  \fill[draw=black] (X\j) circle (2pt); }
\draw (X1) -- (X2);
\draw (X2) -- (X3);
\draw (X3) -- (X4);
\draw (X4) -- (X5);
\draw (X7) -- (X8);
\path (X5)++(140:1cm) coordinate (X51);
\path (X6)++(180:1cm) coordinate (X61);
\draw[dashed] (X5) -- (X51);
\draw[dashed] (X6) -- (X61);
\path (X8)++(360:1cm) coordinate (X81);
\path (X1)++(40:1cm) coordinate (X11);
\draw[dashed] (X8) -- (X81);
\draw[dashed] (X1) -- (X11);
\fill[draw=black] (V) circle (2pt); 
\node at (V1) {$w$}; 
\node[right] at (X1) {$3$}; 
\node[below right] at (X2) {$5^+$}; 
\node[below] at (X3) {$5^+$}; 
\node[below left] at (X4) {$5^+$}; 
\node[left] at (X5) {$3$}; 
\node[above left] at (X6) {$3^+$}; 
\node[above] at (X7) {$5^+$}; 
\node[above right] at (X8) {$3$}; 
\end{tikzpicture}
\end{subfigure}
\caption{Case 6b: $d(w)=8$; subcase where $w$ is on at least five 3-faces.}
\label{fig:T6b}
\end{figure}
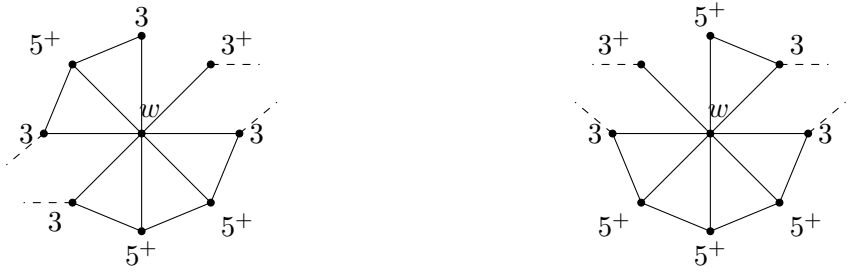

\emph{Case 6c:} $w$ is on exactly five 3-faces that form exactly three maximal runs. \\
By Lemma~\ref{3and4faces}, none of the $4^+$-faces are expensive, so $w$ gives at most $\frac{3}{4}$ to each.
Thus $c^*(w)\geq 10 - 5 \cdot \frac{3}{2} - 3 \cdot \frac{3}{4} >0$.

\smallskip

\textbf{Case 7:} $w$ is a $9^+$-vertex. \\
By Lemma~\ref{adjacenttriangle}, 
the number of 3-faces containing $w$ and a $4^-$-vertex is at most $\floor{2d(w)/3}$.
By (R1), and (R2), $w$ gives at most $\frac{3}{2}$ to each such 3-face, and (R3)--(R8) imply that $w$ gives at most 1 to each other incident face.  If $d(w)=9$, then Lemma~\ref{3and4faces} implies that $w$ cannot be on six expensive 3-faces and three expensive 4-faces so $c^*(w) > 12 - 6 \cdot \frac{3}{2} - 3 \cdot 1 = 0$. If $d(w) \ge 10$, then \[ c^*(w)\geq (2d(w)-6)  - d(w) \cdot 1 - \frac{2d(w)}{3} \cdot \frac{1}{2} = \frac{2d(w)}{3} - 6 > 0. \]

\bigskip
\bigskip 

Finally, we know that 
\[\sum_{x \in V(G) \cup F(G)}\hspace{-16pt} c^*(x) = \sum_{x \in V(G) \cup F(G)}\hspace{-16pt} c(x) \le 0,\] 
and we have shown that every vertex and face ends with nonnegative charge.  Furthermore, we have shown that only vertices of degree 3 or 4 can end with zero charge, so 
we conclude that every vertex of $G$ has degree 3 or 4.
Lemma~\ref{adjacent3and4} now implies that $G$ is 4-regular.
Since every $6^+$-face ends with positive charge, Lemmas~\ref{4vertex3-cycle} and~\ref{4444face} imply that every face of $G$ is a 5-face.
However, then by (R8) every face of $G$ ends with positive charge, a contradiction. 
\end{proof}

Lemma~\ref{lem:discharge} completes the proof of Theorem~\ref{thm:maintorus}. We note that if $G$ is a planar graph, (R9) and Lemma~\ref{4444face} are not needed.


\section{General bounds}\label{sec:rdynamic}

In this section we consider larger genus and larger $r$.  We prove Theorem~\ref{lem:contract}, giving bounds on the $r$-dynamic paintability of graphs in terms of their genus for every $r$.

Given a graph $G$ and edge $uv\in E(G)$, the \emph{weight} of $uv$, denoted $w(uv)$, is $d(u)+d(v)$.
Borodin~\cite{Borodin} proved that planar graphs with minimum degree at least 3 have an edge of weight at most 13.
Ivan\v{c}o extended this to a sharp bound for every orientable surface.

\begin{lemma}[Ivan\v{c}o~\cite{ivanco}]\label{thm:ivanco}
If $G$ be is simple graph with genus $g$ such that $\delta(G)\geq3$, then $G$ has an edge of weight at most
\[ \begin{cases}2g+13&\text{if $0\leq g\leq2$,}\\4g+7&\text{otherwise.}\end{cases}\]
\end{lemma}

We say that a vertex $w$ is \emph{$r$-dull} if fewer than $\min\set{r,d(w)}-1$ distinct colors appear on $N(w)$.

\begin{thmgeneral}
Let $G$ be a graph, and let $g=\gamma(G)$.
\begin{enumerate}
\item  If $g \le 2$ and $r \ge 2g + 11$, then $\xp_r(G) \le (g+5)(r+1)+3$
\item If $g \ge 3$ and $r \ge 4g + 5$, then $\xp_r(G) \le (2g+2)(r+1)+3$.
\end{enumerate}
\end{thmgeneral}
\begin{proof}
We use induction on $|V(G)|$.
If $|V(G)|\le4$, then $G$ is planar and $\xp_r(G)\le |V(G)| < 5r+8$, 
so our base case is complete.

Let 
\begin{align*}
 \ell&=\begin{cases}(g+5)(r+1)+3&\text{if $g\leq2$,}\\(2g+2)(r+1)+3&\text{otherwise,}\end{cases} \\
 \omega&=\begin{cases}2g+13&\text{if $g\leq2$,}\\4g+7&\text{otherwise.}\end{cases}
\end{align*}

Let $f$ be a token assignment to $G$ with $f(v)=\ell$ for all $v\in V(G)$. Suppose $G$ has a $2^-$-vertex $v$. Since $\ell>10$, by Lemma~\ref{2^-vertex} $v$ is reducible and applying induction to $G-v$ completes the proof. Therefore we may assume that $\delta(G)\ge3$.
By Lemma~\ref{thm:ivanco}, there exists an edge $uv\in E(G)$ with weight at most $\omega$. 
Suppose $d(u)\le d(v)$, and let $G'$ be obtained by contracting $uv$.
Let $v$ be the vertex that ``absorbs'' the edges from $u$, deleting multiedges.

Because $G'$ was formed through edge contraction, we have $\gamma(G')\le\gamma(G)$. 
Since the bound for $g = 2$ is less than the bound for $g = 3$,
the induction hypothesis implies that $G'$ is $r$-dynamically $\ell$-paintable.
Since $d_{G'}(v)\leq \omega-2\leq r$ and $d_G(u)\leq\omega-3<r$, $G'$ is \extend{r} to $G$.

Painter colors $v$ according to when $v$ was colored in $G'$. Painter rejects $u$ when any vertex of $N_G(u)$ or $N_G(v)$ is colored. Additionally, 
for each $w\in N_G(u)$, Painter rejects $u$ when $w$ is $r$-dull and a neighbor of $w$ is being colored.
Thus $u$ is rejected at most $d(u)+d(v)-1+(r-1)(d(u)-1)$ times.
We have $d_G(u)\leq(\omega-1)/2$ (since $\omega$ is odd) and $d_G(u)+d_G(v)\leq\omega$.
Simplifying, we get 
\begin{align*}
d(u) + d(v) -1 + (r-1)(d(u)-1) &= (d(u)+d(v)) + (r-1)d(u) -r \\
 &\leq \omega + (r-1)\frac{\omega-1}{2} - r \\
 &= \frac{(\omega-3)(r+1)}{2} + 2 \\
 &= \begin{cases}(g+5)(r+1)+2&\text{if $g\leq2$,}\\(2g+2)(r+1)+2&\text{else}\end{cases} \\
 &= \ell-1.
\end{align*}
Thus Painter rejects $u$ at most $\ell-1$ times, so $G$ is $r$-dynamically $\ell$-paintable.
\end{proof}

Theorem~\ref{lem:contract} is unlikely to be sharp even for the plane.  Hell and Seyffarth~\cite{HS} found examples of planar graphs diameter 2, maximum degree $r$, and $\floor{\frac{3r}{2}}+1$ vertices. For such planar graphs we have $\chi_r(G)=\floor{\frac{3r}{2}}+1$.

\section{Concluding Remarks}

Let $\mathcal{G}_g$ be the family of graphs embeddable on a surface of genus $g$.
Bounds on the $r$-dynamic coloring parameters for graphs of given genus are well known for $r=1$:
for $G\in\mathcal{G}_0$, $\chi_1(G)\le 4$, while $\ch_1(G) \le 5$ and $\xp_1(G)\le5$ with equality achievable for each bound.
%
%
%
Our main results show that, for $G\in\mathcal{G}_1$,
$\chi_3(G)\le\ch_3(G)\le\xp_3(G)\le 10$.
On the torus, equality is achieved by the Petersen graph, but we ask

\begin{question}
Over $G\in\mathcal{G}_0$, what are $\max\chi_3(G)$, $\max\ch_3(G)$, and $\max\xp_3(G)$?
\end{question}

Recall that $\max_{G\in\mathcal{G}_0}\chi_3(G)\ge7$.
Thus answering the question would involve improving the bound in
Corollary~\ref{cor:mainplane} or showing that equality holds there.
Additionally, determining tight upper bounds for $\max_{G\in\mathcal{G}_g}\xp_r(G)$
for $r=3$ and $g>1$ and for all $r>3$ is of interest.

\begin{question}
Except when $r=1$ and $g=0$, is it true that $\max_{G\in\mathcal{G}_g}\chi_r(G)
=\max_{G\in\mathcal{G}_g}\xp_r(G)$?
\end{question}

Note that every nonplanar graph $G$ is 2-dynamically $h(\gamma(G))$-paintable~\cite{mahoney}, which implies equality for $r=2$ and $g>0$.
%
%
%
%
%
%
%

We wish to comment that generalizing our proofs of Theorem~\ref{thm:maintorus} to $r\ge4$ is unlikely. An important tool in our reducibility arguments was that we could add enough edges to the remaining graph to force three colors to appear in the neighborhoods of deleted vertices. For $r=3$, we could do this and keep the degree of each vertex at most its degree in the original graph. For $r\ge4$, we would need to force four colors to appear in the neighborhoods of deleted vertices, which requires more edges. This causes problems with planarity and with guaranteeing the vertices to which edges were added have enough colors in their neighborhoods in the original graph.

\section{Acknowledgments}

We wish to thank Ilkyoo Choi for helpful preliminary discussion and Douglas West for helping to improve the exposition.


%

%
%
%

\bibliographystyle{amsplain}
\bibliography{3dynamicColoringPlanarGraphs}

\end{document}